\documentclass{article}[12pt]

\usepackage{makeidx}

\usepackage[english]{babel}
\usepackage[latin1]{inputenc}
\usepackage[T1]{fontenc}
\usepackage{amsmath, amsthm}
\usepackage{amscd}
\usepackage{amssymb}
\usepackage{amssymb}
\usepackage{latexsym}
\usepackage{amsmath}
\usepackage{amsfonts}
\usepackage{amsthm}

\usepackage{url}

\usepackage{color}

\usepackage{enumerate}
\usepackage{graphicx}
\usepackage{float}

\usepackage{lipsum}
\usepackage{tikz}
\usetikzlibrary{matrix,arrows,decorations.pathmorphing}
\usepackage[english]{babel}
\usepackage[latin1]{inputenc}
\usepackage[T1]{fontenc}
\usepackage{amsmath}
\usepackage{amssymb}
\usepackage{amsthm}
\usepackage{enumerate}
\usepackage{mathtools}

\newtheorem{thm}{Theorem}[section]
\newtheorem{theorem}[thm]{Theorem}
\newtheorem{claim}[thm]{Claim}

\newtheorem{corollary}[thm]{Corollary}
\newtheorem{lemma}[thm]{Lemma}
\newtheorem{proposition}[thm]{Proposition}

\newtheorem{remark}[thm]{Remark}

\theoremstyle{definition}
\newtheorem{example}[thm]{Example}

\makeatletter
\newcommand{\subalign}[1]{%
  \vcenter{%
    \Let@ \restore@math@cr \default@tag
    \baselineskip\fontdimen10 \scriptfont\tw@
    \advance\baselineskip\fontdimen12 \scriptfont\tw@
    \lineskip\thr@@\fontdimen8 \scriptfont\thr@@
    \lineskiplimit\lineskip
    \ialign{\hfil$\m@th\scriptstyle##$&$\m@th\scriptstyle{}##$\crcr
      #1\crcr
    }%
  }
}
\makeatother

\begin{document}

\centerline{\Large \bf Quantum immanants, double Young-Capelli bitableaux}
\bigskip

\centerline{\Large \bf  and }

\bigskip

\centerline{\Large \bf Schur shifted symmetric functions }

\bigskip

\centerline{A. Brini and A. Teolis}
\centerline{\it $^\flat$ Dipartimento di Matematica, Universit\`{a} di
Bologna }
 \centerline{\it Piazza di Porta S. Donato, 5. 40126 Bologna. Italy.}
\centerline{\footnotesize e-mail: andrea.brini@unibo.it}
\medskip

\begin{abstract}
In this paper are introduced two classes of elements in 
the enveloping algebra $\mathbf{U}(gl(n))$: the \emph{double Young-Capelli bitableaux} 
$[\ \fbox{$S \ | \ T$}\ ]$ and the \emph{central} \emph{Schur elements} $\mathbf{S}_{\lambda}(n)$, 
that act in a remarkable way on the highest weight vectors of irreducible Schur modules.

Any element $\mathbf{S}_{\lambda}(n)$ is the sum
of all double Young-Capelli bitableaux  $[\ \fbox{$S \ | \ S$}\ ]$,
$S$    row (strictly) increasing Young tableaux 
of shape $\widetilde{\lambda}$.
The Schur elements $\mathbf{S}_\lambda(n)$ are proved to be the preimages - with respect to the Harish-Chandra isomorphism -
of the  \emph{shifted Schur polynomials}  $s_{\lambda|n}^* \in \Lambda^*(n)$.
Hence, the Schur elements are 
the same as the Okounkov \textit{quantum immanants}, recently described by the present authors as linear  combinations  of  \emph{Capelli immanants}.
This new presentation  of Schur elements/quantum immanants doesn't involve the irreducible characters of symmetric groups.
The Capelli elements $\mathbf{H}_k(n)$ are column Schur elements 
and 
the Nazarov elements $\mathbf{I}_k(n)$ are row Schur elements.
The duality in $\boldsymbol{\zeta}(n)$ 
follows from a 
combinatorial description of the eigenvalues
of the  $\mathbf{H}_k(n)$ on irreducible modules  that is
{\it{dual}} (in the sense of shapes/partitions) to the combinatorial description of the eigenvalues
of the  $\mathbf{I}_k(n)$.

The passage $n \rightarrow \infty$ for the
algebras $\boldsymbol{\zeta}(n)$
is  obtained both as   direct and inverse limit in the category of filtered algebras, 
via the \emph{Olshanski decomposition/projection}.

\end{abstract}


\tableofcontents

\section{Introduction}

The study of  the center
$\boldsymbol{\zeta}(n)$ of the enveloping algebra $\mathbf{U}(gl(n))$ of the
general linear Lie algebra $gl(n, \mathbb{C})$ and  the study of the algebra
$\Lambda^*(n)$ of shifted symmetric polynomials have noble and rather independent origins and motivations.

The theme of central elements  in $\mathbf{U}(gl(n))$ is a standard one in the general theory of Lie algebras, see e.g. \cite{DIX-BR}.
It is an old and actual one, since it
is an  offspring of the celebrated Capelli identity (see, e.g. \cite{Cap1-BR}, \cite{Cap4-BR}, \cite{Howe-BR}, \cite{HU-BR},
\cite{Procesi-BR}, \cite{Umeda-BR}, \cite{Weyl-BR}),
relates to its modern generalizations and applications (see, e.g. \cite{ABP-BR}, \cite{KostantSahi1-BR},
\cite{KostantSahi2-BR}, \cite{MolevNazarov-BR}, \cite{Nazarov-BR}, \cite{Okounkov-BR}, \cite{Okounkov1-BR},
\cite{Sahi3-BR}, \cite{UmedaCent-BR})
as well as to the theory of {\it{Yangians}} (see, e.g.  \cite{Molev1-BR}, \cite{Molev-BR},  \cite{Nazarov2-BR}).

The algebra $\Lambda^*(n)$ of shifted symmetric polynomials is a remarkable deformation of the  algebra
 $\Lambda(n)$ of symmetric polynomials and its study fits into the mainstream of generalizations of the classical theory
(see, e.g. {\it{factorial symmetric functions}}, \cite{BL1-BR}, \cite{BL2-BR}, \cite{CL-BR}, \cite{GG-BR}, 
\cite{GH-BR}, \cite{M2-BR}, \cite{M1-BR}).

Since the algebras $\boldsymbol{\zeta}(n)$ and $\Lambda^*(n)$
are related by the Harish-Chandra isomorphism $\chi_n$ (see, e.g. \cite{OkOlsh-BR}),
their investigation can be essentially regarded as a  single topic,
and this fact gave rise to a fruitful interplay between representation-theoretic
 methods (e.g., eigenvalues on irreducible representations)
and combinatorial techniques (e.g., generating functions).

In this work, we propose a new approach to a systematic study of some of the main features of the theory of the center
$\boldsymbol{\zeta}(n)$ of $\mathbf{U}(gl(n))$ and of the algebra $\Lambda^*(n)$ of shifted symmetric polynomials
that allows the whole theory to be developed, in a transparent and concise way, from a combinatorial 
representation theoretic point of view, that is entirely
in the center $\boldsymbol{\zeta}(n)$.

The paper is organized as follows.

In  Section \ref{sec 3}, we provide a synthetic presentation   
of the {\it{superalgebraic method of virtual variables}} 
for $gl(n)$. For details, we refer the reader to \cite{BriUMI-BR}, \cite{BT-jaa} and \cite{BT-CA}.
This method was developed  by the present authors
for the   general linear Lie superalgebras $gl(m|n)$ (see, e.g. \cite{KAC1-BR}, \cite{Scheu-BR})  
in the series of notes
\cite{Bri-BR}, \cite{BriUMI-BR}, \cite{Brini1-BR}, \cite{Brini2-BR}, \cite{Brini3-BR},
\cite{Brini4-BR}, \cite{BRT-BR}.
The technique of virtual variables is  an extension of Capelli's method of  {\textit{ variabili ausilarie}}
(Capelli \cite{Cap4-BR}, see also Weyl \cite{Weyl-BR}).

The superalgebraic method of virtual variables allows us to express remarkable classes 
of elements in $\mathbf{U}(gl(n))$, namely,

\begin{itemize}

\item [--] the \emph{Capelli bitableaux} $[S|T] \in \mathbf{U}(gl(n))$

\item [--] the \emph{Young-Capelli bitableaux} $[S| \fbox{$T$}] \in {\mathbf{U}}(gl(n))$

\item [--] the \emph{double Young-Capelli bitableaux} $[\ \fbox{$S \ | \ T$}\ ] \in {\mathbf{U}}(gl(n))$
\end{itemize}
as the images - with respect to the $Ad_{gl(n)}$-adjoint equivariant  Capelli 
\emph{devirtualization epimorphism} - of simple expressions in an enveloping 
superalgebra $\mathbf{U}(gl(m_0|m_1+n))$.

Capelli (determinantal) bitableaux are generalizations of  the famous  
\emph{determinantal elements} in ${\mathbf{U}}(gl(n))$ 
introduced by Capelli in  $1887$ \cite{Cap1-BR} (see, e.g. \cite{BT-CA}).
Young-Capelli bitableaux were introduced by the present authors 
several years ago \cite{Brini2-BR}, \cite{Brini3-BR}, \cite{Brini4-BR} and might be regarded
as generalizations of the Capelli determinantal elements in ${\mathbf{U}}(gl(n))$ 
as well as of the \emph{Young symmetrizers}
of the classical representation theory of symmetric groups  (see, e.g. \cite{Weyl-BR}).
Double Young-Capelli bitableaux are essentially new and play a crucial role in the present paper.

In plain words, the Young-Capelli bitableau $[S| \fbox{$T$}]$ is obtained 
by adding a \emph{column symmetrization} to  the Capelli bitableau $[S|T]$ 
(Proposition \ref{YC as sym of Capelli}) and the double Young-Capelli bitableau
$[\ \fbox{$S \ | \ T$}\ ]$ is obtained by adding a further \emph{row skew-symmetrization}
to the Young-Capelli bitableau $[S| \fbox{$T$}]$ (Proposition \ref{double YC from YC}).

In Section \ref{Schur}, we regard the \emph{supersymmetric} superalgebra ${\mathbb C}[M_{m_0|m_1+n,d}]$ 
as  $\mathbf{U}(gl(m_0|m_1+n))$-module and define the {\textit{Schur supermodules}} $Schur_\lambda(m_0,m_1+n)$ 
as (irreducible) submodules of ${\mathbb C}[M_{m_0|m_1+n,d}]$.
Schur supermodules are isomorphic to the modules constructed by Berele and Regev \cite{Berele1-BR}, \cite{Berele2-BR} as tensor modules induced by Young symmetrizers
(see, e.g. \cite{Weyl-BR}) when they act by a ``signed action'' of the symmetric group (see also  \cite{King-BR}).
The description presented here
is simpler than the tensor description, and provides a close connection with the superstraightening theory of Grosshans, Rota and Stein \cite{rota-BR}.
The classical Schur modules $Schur_\lambda(n)$ - $gl(n)$-irreducible modules
with (nonnegative) integral highest weights - are here regarded as $gl(n)$-submodules of suitable
Schur supermodules. The crucial and new result  is that double
Young-Capelli bitableaux act in a remarkable way on the highest weight vectors of Schur modules
(subsection  \ref{sect Double YC action}, Theorem \ref{Double YC action}).

In Section \ref{The center}, we consider a class of central 
elements $\mathbf{S}_{\lambda}(n)$, $\tilde{\lambda}_1 \leq n$,
 which arise in a natural way in the context of the virtual method when dealing 
with symmetry and skew-symmetry in $\mathbf{U}(gl(n))$.

These elements are expressed as linear combinations
\begin{equation}\label{schur second equation}
\mathbf{S}_\lambda(n) =
\frac {1} {H(\tilde{\lambda})} \  \ \sum_S \ [\ \fbox{$S \ | \ S$}\ ]  \in {\mathbf{U}}(gl(n))
\end{equation}
of double Young-Capelli bitableaux where the sum is extended to all  row (strictly) increasing tableaux 
$S$ of shape $sh(S) = \widetilde{\lambda}$.

We call the elements $\mathbf{S}_{\lambda}(n)$ the \emph{Schur elements}. 

The main results are Theorems \ref{Schur action} and \ref{Vanishing theorem} that
follow from Theorem \ref{Double YC action} and provide notable descriptions of the action
of the central Schur elements $\mathbf{S}_{\lambda}(n)$ on the highest weight vectors of Schur modules.
Theorem \ref{Schur action} implies that the set of the elements $\mathbf{S}_\lambda(n)$,
$\widetilde{\lambda}_1 \leq n$ is a basis of the center $\boldsymbol{\zeta}(n)$.

By combining Theorem \ref{Schur action} with the \emph{Sahi-Okounkov characterization 
Theorem} (\cite{Sahi2-BR}, \cite{OkOlsh-BR}, \cite{Okounkov-BR}, here quoted as Proposition \ref{SahiOkounkv}), 
we infer that
the Schur elements $\mathbf{S}_\lambda(n)$ are the preimages - with respect to the Harish-Chandra isomorphism -
 of the  \emph{shifted Schur polynomials}  $s_{\lambda|n}^* \in \Lambda^*(n)$  \cite{Sahi1-BR},  
\cite{OkOlsh-BR}.
Hence, the Schur elements are 
the same as the  \textit{quantum immanants}, first presented by  Okounkov as traces of
\textit{fusion matrices} (\cite{Okounkov-BR}, \cite{Okounkov1-BR}) and, recently, described by the present authors as linear  combinations (with explicit coefficients) of ``diagonal'' \emph{Capelli immanants}, see \cite{BT-jaa}.
Presentation (\ref{schur second equation}) of Schur elements/quantum immanants doesn't involve the irreducible characters of symmetric groups.
Furthermore, it is better suited to the study  of the duality in the algebra $\boldsymbol{\zeta}(n)$
as well as to the study of the limit $n \rightarrow \infty$.

We examine two further classes of central elements, namely, the \textit{determinantal Capelli elements}
$\mathbf{H}_k(n), \quad k = 1, 2, \ldots, n$ (see, e.g. \cite{Cap1-BR}, \cite{Cap2-BR}, \cite{Cap3-BR}, 
 \cite{BriUMI-BR}),
and the \textit{permanental Nazarov elements} $\mathbf{I}_k(n), \quad k \in \mathbf{Z}^+$
(see, e.g.  \cite{Nazarov-BR}, \cite{Nazarov2-BR}, \cite{UmedaCent-BR}, \cite{UmedaHirai-BR}, 
see also \cite{BriUMI-BR}),
which provide two systems of algebra generators of the center $\boldsymbol{\zeta}(n)$.

The Capelli elements $\mathbf{H}_k(n)$ are column Schur elements, specifically,
$$
\mathbf{H}_k(n) = \mathbf{S}_{(1^k)}(n).
$$

The Nazarov elements $\mathbf{I}_k(n)$ are row Schur elements, specifically,
$$
\mathbf{I}_k(n) = \mathbf{S}_{(k)}(n).
$$

The duality in $\boldsymbol{\zeta}(n)$ (Theorem \ref{finite duality}) 
immediately follows from a 
combinatorial description of the eigenvalues
of the Capelli elements $\mathbf{H}_k(n)$ on irreducible Schur modules (Proposition \ref{horizontal strip}) 
that is {\it{dual}} (in the sense of shapes/partitions) to the combinatorial description of the eigenvalues
of the Nazarov elements $\mathbf{I}_k(n)$ (Theorem \ref{vertical strip}.1).

The passage to the infinite dimensional case $n \rightarrow \infty$ for the
algebras $\boldsymbol{\zeta}(n)$ is rather subtle;  the ``naive''
$\infty$-dimensional analogue of the algebras $\mathbf{U}(gl(n))$,
that is the direct limit algebra $\underrightarrow{lim} \ \mathbf{U}(gl(n))$ with respect to the {\it{``inclusion''  monomorphisms}},
has trivial center.
In Section \ref{zeta limit}, the $\infty$-dimensional analogue $\boldsymbol{\zeta}$ of the algebras $\boldsymbol{\zeta}(n)$
is  obtained as the  \emph{direct limit algebra} $\underrightarrow{lim} \ \boldsymbol{\zeta}(n)$ (in the category of filtered algebras)
with respect to the family of monomorphisms $\mathbf{i}_{n+1,n} : \boldsymbol{\zeta}(n) \hookrightarrow \boldsymbol{\zeta}(n+1)$, where
$$
\mathbf{i}_{n+1,n} \big( \mathbf{H}_k(n)  \big) = \mathbf{H}_k(n+1), \quad k = 1, 2, \ldots, n.
$$

An intrinsic/invariant presentation of the  monomorphisms $\mathbf{i}_{n+1,n}$
is obtained, in subsection \ref{The Olshanski decomposition/projection}, via the
{\it{Olshanski projections}} 
$\boldsymbol{\mu}_{n,n+1} : \boldsymbol{\zeta}(n+1) \twoheadrightarrow \boldsymbol{\zeta}(n)$ 
\cite{Olsh1-BR}, \cite{Olsh3-BR} (see also Molev \cite{Molev1-BR}). The Olshanski projections
$\boldsymbol{\mu}_{n,n+1}$ are {\it{left}} inverses of the  monomorphisms $\mathbf{i}_{n+1,n}$,
 and they become {\it{two-sided}}
inverses when restricted to the filtration elements $\boldsymbol{\zeta}(n+1)^{(m)}$ 
and $\boldsymbol{\zeta}(n)^{(m)}$,
for $n$ sufficiently large (Propositions \ref{inverso filtrato} and  Proposition  \ref{pi=mu}).
The  interplay between the monomorphisms $\mathbf{i}_{n+1,n}$ and the projections $\boldsymbol{\mu}_{n,n+1}$
shows the algebra $\boldsymbol{\zeta}$
admits a double presentation, both as a direct limit and as an inverse limit.

Since the Olshanski projection
$\boldsymbol{\mu}_{n,n+1}$ maps $\mathbf{H}_k(n+1)$ to $\mathbf{H}_k(n)$,
$\mathbf{I}_k(n+1)$ to $\mathbf{I}_k(n)$ and $\mathbf{S}_\lambda(n+1)$ to
$\mathbf{S}_\lambda(n)$, then 
$$
\mathbf{i}_{n+1,n}  \big(\mathbf{H}_k(n)\big) = \mathbf{H}_k(n+1),
\qquad
\mathbf{i}_{n+1,n}  \big(\mathbf{I}_k(n)\big) = \mathbf{I}_k(n+1),
$$
and
$$
\mathbf{i}_{n+1,n}  \big(\mathbf{S}_\lambda(n)\big) = \mathbf{S}_\lambda(n+1),
$$
for $n$ sufficiently large.

Hence, the direct limits
$$
\mathbf{H}_k  \stackrel{def}{=} \underrightarrow{lim}  
\ \mathbf{H}_k(n) \in \boldsymbol{\zeta}
\qquad
\mathbf{I}_k  \stackrel{def}{=} \underrightarrow{lim}  
\ \mathbf{I}_k(n) \in \boldsymbol{\zeta}
$$
and
$$
\mathbf{S}_\lambda  \stackrel{def}{=} \underrightarrow{lim}  
\ \mathbf{S}_\lambda(n) \in \boldsymbol{\zeta}
$$
can be consistently  written as  {\it{formal series}}
by naturally extending to infinite sums the  finite sums 
(eqs. (\ref{Schur basis2}), (\ref{Capelli elements virtual}), (\ref{multiset}))
that define
$\mathbf{H}_k(n)$, $\mathbf{I}_k(n)$ and $\mathbf{S}_\lambda(n)$,
respectively.

The algebra $\boldsymbol{\zeta}$ is
isomorphic to the algebra $\Lambda^*$ of {\it{shifted symmetric functions}} (Theorem \ref{isomorfismo HC infinito}).
The Olshanski projections are the
natural counterpart, in the context of the centers $\boldsymbol{\zeta}(n)$, of the Okounkov-Olshanski {\it{stability principle}}
for the algebras $\Lambda^*(n)$ of shifted symmetric polynomials \cite{OkOlsh-BR}, the isomorphism
$\chi : \boldsymbol{\zeta} \rightarrow \Lambda^*$ is indeed the ``limit'' of the 
Harish-Chandra isomorphisms $\chi_n : \boldsymbol{\zeta}(n) \rightarrow \Lambda^*(n)$ and it admits a transparent representation-theoretic interpretation (Sections  \ref{Lambda(n)} and \ref{Lambda}).

\section{A glimpse on the superalgebraic method of virtual variables }\label{sec 3}

\subsection{The superalgebras $gl(m_0|m_1+n)$ and ${\mathbb C}[M_{m_0|m_1+n,d}]$}

\subsubsection{The general linear Lie super algebra $gl(m_0|m_1+n)$}\label{superalgebra first}
 Given a vector space $V_n$ of dimension $n$, we will regard it as a subspace of a $
\mathbb{Z}_2-$graded vector space
 $W = W_0 \oplus W_1$, where
$$
W_0 = V_{m_0}, \qquad W_1 = V_{m_1} \oplus V_n.
$$
The  vector spaces
$V_{m_0}$ and $V_{m_1}$ (we assume that 
$dim(V_{m_0})=m_0$ and $dim(V_{m_1})=m_1$ are ``sufficiently large'') are called
the  {\textit{positive virtual (auxiliary)
vector space}},  the {\textit{negative virtual (auxiliary) vector space}}, respectively, and $V_n$ 
is called the {\textit{(negative) proper vector space}}.

 The inclusion $V_n \subset W$ induces a natural embedding of the ordinary general 
linear Lie algebra $gl(n)$ of $V_n$ into the
 {\textit{auxiliary}}
general linear Lie {\it{superalgebra}} $gl(m_0|m_1+n)$ of $W = W_0 \oplus W_1$ (see, e.g. \cite{KAC1-BR}, 
\cite{Scheu-BR}).

Let
$
\mathcal{A}_0 = \{ \alpha_1, \ldots, \alpha_{m_0} \},$  $\mathcal{A}_1 = \{ \beta_1, \ldots, \beta_{m_1} \},$
$\mathcal{L} = \{ 1, 2,  \ldots, n \}$
denote \emph{fixed  bases} of $V_{m_0}$, $V_{m_1}$ and $V_n$, respectively; 
therefore $|\alpha_s| = 0 \in \mathbb{Z}_2,$
and $|\beta_t| = |i|   = 1 \in \mathbb{Z}_2.$

Let
$$
\{ e_{a, b}; a, b \in \mathcal{A}_0 \cup \mathcal{A}_1\cup \mathcal{L}  \}, \qquad |e_{a, b}| =
|a|+|b| \in \mathbb{Z}_2
$$
be the standard $\mathbb{Z}_2-$homogeneous basis of the Lie superalgebra $gl(m_0|m_1+n)$ provided by the
elementary matrices. The elements $e_{a, b} \in gl(m_0|m_1+n)$ are $\mathbb{Z}_2-$homogeneous of
$\mathbb{Z}_2-$degree $|e_{a, b}| = |a| + |b|.$

The superbracket of the Lie superalgebra $gl(m_0|m_1+n)$ has the following explicit form:
$$
\left[ e_{a, b}, e_{c, d} \right] = \delta_{bc} \ e_{a, d} - (-1)^{(|a|+|b|)(|c|+|d|)} \delta_{ad}  \ e_{c, b},
$$
$a, b, c, d \in \mathcal{A}_0 \cup \mathcal{A}_1 \cup \mathcal{L} .$

In the following, the elements of the sets $\mathcal{A}_0, \mathcal{A}_1, \mathcal{L} $ will be called
\emph{positive virtual symbols}, \emph{negative virtual symbols} and \emph{negative proper symbols},
respectively.

\subsubsection{The  supersymmetric algebra ${\mathbb C}[M_{m_0|m_1+n,d}]$}

Let
$$
{\mathbb C}[M_{n,d}] =    {\mathbb C}[(i|j)]_{i=1,\ldots,n, j=1,\ldots,d}
$$
be the polynomial algebra in the (commutative)
 entries $(i|j)$ of the matrix:
$$
M_{n,d} = \left[ (i|j) \right]_{i=1,\ldots,n, j=1, \ldots,d}=
 \left(
 \begin{array}{ccc}
 (1|1) & \ldots & (1|d) \\
 \vdots  &        & \vdots \\
 (n|1) & \ldots & (n|d) \\
 \end{array}
 \right).
$$

We regard the commutative algebra ${\mathbb C}[M_{n,d}]$
as a subalgebra of the \textit{``auxiliary'' supersymmetric algebra}
$$
{\mathbb C}[M_{m_0|m_1+n,d}] 
$$
generated by the ($\mathbb{Z}_2$-graded) variables 
$$
(a|j), \quad a \in  \mathcal{A}_0\cup \mathcal{A}_1\cup \mathcal{L} , 
\quad j \in \mathcal{P}  = \{j=1, \ldots,d; |j|=1 \in \mathbb{Z}_2 \},
$$
with $|(a|j)| =  |a|+|j| \in \mathbb{Z}_2 $,
subject to the commutation relations:
$$
(a|h)(b|k) = (-1)^{|(a|h)||(b|k)|} \ (b|k)(a|h).
$$
In plain words, ${\mathbb C}[M_{m_0|m_1+n,d}]$ is the free supersymmetric algebra 
$$
{\mathbb C}\big[ (\alpha_s|j), (\beta_t|j), (i|j) \big]
$$
 generated by the ($\mathbb{Z}_2$-graded) variables $(\alpha_s|j), (\beta_t|j), (i|j)$,
$j = 1, 2, \ldots, d$,
where all the variables commute each other, with the exception of
pairs of variables $(\alpha_s|j), (\alpha_t|j)$ that skew-commute:
$$
(\alpha_s|j) (\alpha_t|j) = - (\alpha_t|j) (\alpha_s|j).
$$

In the standard notation of multilinear algebra, we have:
\begin{align*}
{\mathbb C}[M_{m_0|m_1+n,d}]
& \cong \Lambda \big[ W_0 \otimes P_d \big]
\otimes      {\mathrm{Sym}} \big[ W_1  \otimes P_d \big] \\
 & =
 \Lambda \big[ V_{m_0} \otimes P_d \big]
\otimes      {\mathrm{Sym}} \big[ (V_{m_1} \oplus V_n)  \otimes P_d \big]
\end{align*}
where $P_d = (P_d)_1$ denotes the trivially  $\mathbb{Z}_2-$graded  vector space with distinguished basis 
$\mathcal{P}  = \{j=1, \ldots,d; |j|=1 \in \mathbb{Z}_2 \}.$

\subsubsection{Left superderivations and left superpolarizations}

A {\it{left superderivation}} $D^{\textit{l}}$ ($\mathbb{Z}_2-$homogeneous of degree $|D^{\textit{l}}|$) (see, e.g. \cite{Scheu-BR}, \cite{KAC1-BR}) on
${\mathbb C}[M_{m_0|m_1+n,n}]$ is an element of the superalgebra $End_\mathbb{C}[\mathbb{C}[M_{m_0|m_1+n,d}]]$
that satisfies "Leibniz rule"
$$
D^{\textit{l}}(\textbf{p} \cdot \textbf{q}) = D^{\textit{l}}(\textbf{p}) \cdot \textbf{q} + 
(-1)^{|D^{\textit{l}}||\textbf{p}|} \textbf{p} \cdot D^{\textit{l}}(\textbf{q}),
$$
for every $\mathbb{Z}_2-$homogeneous of degree $|\textbf{p}|$ element $\textbf{p} \in \mathbb{C}[M_{m_0|m_1+n,d}].$

Given two symbols $a, b \in \mathcal{A}_0                                               \cup \mathcal{A}_1                        \cup \mathcal{L} $, the {\textit{left superpolarization}} $D^{\textit{l}}_{a,b}$ 
of $b$ to $a$
is the unique {\it{left}} superderivation of ${\mathbb C}[M_{m_0|m_1+n,n}]$ of $\mathbb{Z}_2-$degree 
$|D^{\textit{l}}_{a,b}| = |a| + |b| \in \mathbb{Z}_2$ such that
$$
D^{\textit{l}}_{a,b} \left( (c|j) \right) = \delta_{bc} \ (a|j), \ c \in \mathcal{A}_0                                               \cup \mathcal{A}_1    \cup \mathcal{L} , \ j = 1, \ldots, n.
$$

Informally, we say that the operator $D^{\textit{l}}_{a,b}$ {\it{annihilates}} the symbol $b$ 
and {\it{creates}} the symbol $a$.

\subsubsection{The superalgebra ${\mathbb C}[M_{m_0|m_1+n,n}]$ as a $\mathbf{U}(gl(m_0|m_1+n))$-module}

Since
$$
D^{\textit{l}}_{a,b}D^{\textit{l}}_{c,d} -(-1)^{(|a|+|b|)(|c|+|d|)}D^{\textit{l}}_{c,d}D^{\textit{l}}_{a,b} =
\delta_{b,c}D^{\textit{l}}_{a,d} -(-1)^{(|a|+|b|)(|c|+|d|)}\delta_{a,d}D^{\textit{l}}_{c,b},
$$
the map
$$
e_{a,b} \mapsto D^{\textit{l}}_{a,b}, \qquad a, b \in \mathcal{A}_0                                               \cup \mathcal{A}_1                        \cup \mathcal{L} 
$$
is a Lie superalgebra morphism from $gl(m_0|m_1+n)$ to $End_\mathbb{C}\big[\mathbb{C}[M_{m_0|m_1+n,d}]\big]$
and, hence, it uniquely defines a
representation:
$$
\varrho : \mathbf{U}(gl(m_0|m_1+n)) \rightarrow End_\mathbb{C}[\mathbb{C}[M_{m_0|m_1+n,d}]],
$$
where $\mathbf{U}(gl(m_0|m_1+n))$ is the enveloping superalgebra of $gl(m_0|m_1+n)$.

In the following, we always regard the superalgebra $\mathbb{C}[M_{m_0|m_1+n,d}]$ as a $\mathbf{U}(gl(m_0|m_1+n))-$supermodule,
with respect to the action induced by the representation $\varrho$:
$$
e_{a,b} \cdot \mathbf{p} = D^{\textit{l}}_{a,b}(\mathbf{p}),
$$
for every $\mathbf{p} \in {\mathbb C}[M_{m_0|m_1+n,n}].$

We recall that  $\mathbf{U}(gl(m_0|m_1+n))-$module  $\mathbb{C}[M_{m_0|m_1+n,d}]$
is  a semisimple module, whose simple submodules are - up to isomorphism - {\it{Schur supermodules}} 
(see, e.g. \cite{Brini1-BR}, \cite{Brini2-BR}, \cite{Bri-BR}. For a more traditional presentation, see also 
\cite{CW-BR}).

Clearly, $\mathbf{U}(gl(0|n)) = \mathbf{U}(gl(n))$ is a subalgebra of $\mathbf{U}(gl(m_0|m_1+n))$
and the subalgebra $\mathbb{C}[M_{n,d}]$ is a $\mathbf{U}(gl(n))-$submodule of  $\mathbb{C}[M_{m_0|m_1+n,d}]$.

\subsubsection{The virtual algebra $Virt(m_0+m_1,n)$ and the virtual
presentations of elements in $\mathbf{U}(gl(n))$}

We say that a product
$$
e_{a_mb_m} \cdots e_{a_1b_1} \in \mathbf{U}(gl(m_0|m_1+n)), 
\quad a_i, b_i \in \mathcal{A}_0 \cup \mathcal{A}_1 \cup \mathcal{L} , \ i= 1, \ldots, m
$$
is an {\textit{irregular expression}} whenever
  there exists a right subword
$$e_{a_i,b_i} \cdots e_{a_2,b_2} e_{a_1,b_1},$$
$i \leq m$ and a
virtual symbol $\gamma \in \mathcal{A}_0 \cup \mathcal{A}_1$ such that
\begin{equation}\label{irrexpr-BR}
 \# \{j;  b_j = \gamma, j \leq i \}  >  \# \{j;  a_j = \gamma, j < i \}.
\end{equation}

The meaning of an irregular expression in terms of the action of  $\mathbf{U}(gl(m_0|m_1+n))$  
by left superpolarization on
the algebra $\mathbb{C}[M_{m_0|m_1+n,d}]$ is that there exists a
virtual symbol $\gamma$ and a right subsequence in which the symbol $\gamma$ is \emph{annihilated} 
more times than it was already \emph{created} and, therefore, the action of an irregular expression
on the algebra $\mathbb{C}[M_{n,d}]$ is \emph{zero}. 

\begin{example}
Let $\gamma \in  \mathcal{A}_0 \cup \mathcal{A}_1$ and $x_i, x_j \in \mathcal{L}.$ The product
$$
e_{\gamma,x_j} e_{x_i,\gamma} e_{x_j,\gamma} e_{\gamma,x_i}
$$
is an irregular expression.
\end{example}\qed

Let $\mathbf{Irr}$   be
the {\textit{left ideal}} of $\mathbf{U}(gl(m_0|m_1+n))$ generated by the set of
irregular expressions.

\begin{proposition}
The superpolarization action
of any element of $\mathbf{Irr}$ on the subalgebra $\mathbb C[M_{n,d}] \subset \mathbb{C}[M_{m_0|m_1+n,d}]$ - via the representation $\varrho$ -
is identically zero.
\end{proposition}

\begin{proposition}
The sum ${\mathbf{U}}(gl(0|n)) + \mathbf{Irr}$ is a direct sum of vector subspaces of $\mathbf{U}(gl(m_0|m_1+n)).$
\end{proposition}

\begin{proposition}
The direct sum vector subspace $\mathbf{U}(gl(0|n)) \oplus \mathbf{Irr}$ is a \emph{subalgebra} 
of $\mathbf{U}(gl(m_0|m_1+n)).$
\end{proposition}

The subalgebra
$$
Virt(m_0+m_1,n) = \mathbf{U}(gl(0|n)) \oplus \mathbf{Irr} \subset {\mathbf{U}}(gl(m_0|m_1+n)).
$$
is called the {\textit{virtual algebra}}.

\begin{proposition}
The left ideal  $\mathbf{Irr}$ of ${\mathbf{U}}(gl(m_0|m_1+n))$
is a two sided ideal of $Virt(m_0+m_1,n).$
\end{proposition}

The {\textit{Capelli devirtualization epimorphism}} is the surjection
$$
\mathfrak{p} : Virt(m_0+m_1,n) = \mathbf{U}(gl(0|n)) \oplus \mathbf{Irr} \twoheadrightarrow \mathbf{U}(gl(0|n)) = \mathbf{U}(gl(n))
$$
with $Ker(\mathfrak{p}) = \mathbf{Irr}.$

Any element in $\textbf{M} \in Virt(m_0+m_1,n)$ defines an element in
$\textbf{m} \in \mathbf{U}(gl(n))$ - via the map $\mathfrak{p}$ -
 and $\textbf{M}$ is called a \textit{virtual
presentation} of $\textbf{m}$.

Furthermore,
\begin{proposition}
The subalgebra $\mathbb C[M_{n,d}] \subset \mathbb{C}[M_{m_0|m_1+n,d}]$ is  invariant with respect to the action
of the subalgebra
$
Virt(m_0+m_1, n).
$
\end{proposition}
\begin{proposition}\label{virtual action}
For every element $\mathbf{m} \in {\mathbf{U}}(gl(n))$, the action of $\mathbf{m}$ on 
the subalgebra $\mathbb C[M_{n,d}]$
is the same of the action of any of its virtual presentation $\mathbf{M}  \in Virt(m_0+m_1,n).$
In symbols, 
$$
if  \quad
\mathfrak{p}(\mathbf{M}) = \mathbf{m}
\quad
then
\quad
\mathbf{m} \cdot \mathbf{P} = \mathbf{M} \cdot \mathbf{P},
\quad
for \ every \ \mathbf{P} \in \mathbb C[M_{n,d}].
$$
\end{proposition}
Since the map $\mathfrak{p}$  a surjection, any element
$\mathbf{m} \in \mathbf{U}(gl(n))$ admits several virtual
presentations. In the sequel, we even take virtual presentations
as the \emph{definition} of special elements in $\mathbf{U}(gl(n)),$ 
and this method will turn out to be quite effective.

The superalgebra ${\mathbf{U}}(gl(m_0|m_1+n))$ is a Lie module with respect 
to the adjoint representation $Ad_{gl(m_0|m_1+n)}$. Since $gl(n) = gl(0|n)$ 
is a Lie subalgebra of $gl(m_0|m_1+n)$, then ${\mathbf{U}}(gl(m_0|m_1+n))$ is a $gl(n)-$module 
with respect to the adjoint action $Ad_{gl(n)}$ of $gl(n)$.
We recall a couple of results from \cite{BT-CA}.

\begin{proposition} The virtual algebra $Virt(m_0+m_1,n)$ is a submodule 
of ${\mathbf{U}}(gl(m_0|m_1+n))$ with respect to the adjoint action $Ad_{gl(n)}$ of $gl(n)$.
\end{proposition}

\begin{proposition}\label{rappresentazione aggiunta-BR} The Capelli  epimorphism 
$$
\mathfrak{p} : Virt(m_0+m_1,n)  \twoheadrightarrow \mathbf{U}(gl(n))
$$ is an $Ad_{gl(n)}-$\emph{equivariant} map.
\end{proposition}

\begin{corollary}\label{centrality gen} The isomorphism $\mathfrak{p}$ maps
any  $Ad_{gl(n)}-$invariant element $\mathbf{m} \in Virt(m_0+m_1,n)$  to a \emph{central} 
element of $\mathbf{U}(gl(n))$.
\end{corollary}

\textit{Balanced monomials} are  elements of the algebra ${\mathbf{U}}(gl(m_0|m_1+n))$
 of the form:
\begin{itemize}\label{defbalanced monomials-BR}
\item [--] $e_{{i_1},\gamma_{p_1}} \cdots e_{{i_k},\gamma_{p_k}} \cdot
e_{\gamma_{p_1},{j_1}} \cdots e_{\gamma_{p_k},{j_k}},$
\item [--]
$e_{{i_1},\theta_{q_1}} \cdots e_{{i_k},\theta_{q_k}} \cdot
e_{\theta_{q_1},\gamma_{p_1}} \cdots e_{\theta_{q_k},\gamma_{p_k}} \cdot
e_{\gamma_{p_1},{j_1}} \cdots e_{\gamma_{p_k},{j_k}},$
\item [--] and so on,
\end{itemize}
where
$i_1, \ldots, i_k, j_1, \ldots, j_k \in L,$
i.e., the $i_1, \ldots, i_k, j_1, \ldots, j_k$ are $k$
proper (negative) symbols, and the
$\gamma_{p_1}, \ldots, \gamma_{p_k}, \ldots, \theta_{q_1}, \ldots, \theta_{q_k}, \ldots$ are
virtual symbols.
In plain words, a balanced monomial is product of two or more factors  where the
rightmost one  \textit{annihilates} (by superpolarization)
the $k$ proper symbols $ j_1, \ldots, j_k$ and
\textit{creates} (by superpolarization) some virtual symbols;
 the leftmost one  \textit{annihilates} all the virtual symbols
and \textit{creates} the $k$ proper symbols $ i_1, \ldots, i_k$;
between these two factors, there might be further factors that annihilate
 and create  virtual symbols only.

\begin{proposition}
Every balanced monomial belongs to $Virt(m_0+m_1,n)$. Hence,
the Capelli epimorphism $\mathfrak{p}$ maps  balanced monomials to elements of $\mathbf{U}(gl(n)).$
\end{proposition}

\subsection{Four special classes of elements in $Virt(m_0+m_1,n)$ and their images in ${\mathbf{U}}(gl(n))$}\label{Capbit}

We will introduce four classes of remarkable elements of the enveloping algebra ${\mathbf{U}}(gl(n))$, that
we call {\textit{bitableaux monomials}}, {\textit{Capelli bitableaux}}, {\textit{Young-Capelli bitableaux}} 
and {\textit{double Young-Capelli bitableaux}}, respectively.

\subsubsection{Partitions and Young tableaux}\label{comb Young tab}

 Let $\lambda = (\lambda_1 \geq  \lambda_2 \geq \cdots \geq \lambda_p) \vdash h$ be a partition
 of the positive integer $h \in \mathbb{Z}^+$, where $p = \it{l}(\lambda)$
is the  \emph{length} of $\lambda$.

We denote by $\tilde{\lambda} = (\tilde{\lambda}_1 \geq \tilde{\lambda}_2 \geq \cdots \geq \tilde{\lambda}_q)$
the \emph{conjugate partition} of $\lambda$, that is
$$
\tilde{\lambda}_i = \# \{ j = 1, 2, \ldots ,p; \lambda_j \geq i \}, \quad i = 1, 2, \ldots \lambda_1;
$$ 
clearly, $\it{l}(\tilde{\lambda}) = \lambda_1.$
 
Label the boxes of the Ferrers diagram of the partition $\lambda$
with the numbers $1, 2, \ldots , h$ in the following way:
$$
\begin{array}{lllll}
1 & 2 &  \cdots & \cdots & \lambda_1 \\
\lambda_1 + 1 &  \lambda_1 + 2 & \cdots & \lambda_1 + \lambda_2 &    \\
\cdots & \cdots & \cdots &  &  \\
\cdots & \cdots & h &  &  \\
\end{array}.
$$
A {\textit{Young tableau}} $T$ of shape $\lambda$ over the alphabet 
$\mathcal{A} =  \{a_1, a_2, \ldots \}$ 
is a map $T : \underline{h} = \{1, 2, \ldots , h \}  \rightarrow \mathcal{A}$; the element $T(i)$
is the symbol in the cell $i$ of the tableau $T$.

The sequences
$$
\begin{array}{l}
T(1)  T(2) \cdots  T(\lambda_1),
\\
T(\lambda_1 + 1)  T(\lambda_1 + 2)  \cdots  T(\lambda_1 + \lambda_2),
\\
 \ldots \ldots
\end{array}
$$
are called the {\textit{row words}} of the Young tableau $T$.

We will also denote a Young tableau
by its sequence of rows words, that is $T = (\omega_1, \omega_2, \ldots, \omega_p)$.
Furthermore, the {\textit{word of the tableau}} $T$ is the concatenation
\begin{equation}\label{word}
w(T) = \omega_1\omega_2 \cdots \omega_p.
\end{equation}

The {\textit{content}} of a tableau $T$ is the function $c_T : \mathcal{A} \rightarrow \mathbb{N}$,
$$
c_T(a) = \sharp \{i \in \underline{h}; \ T(i) = a \}.
$$

Set
\begin{equation}\label{Deruyts}
D_{\lambda} = \left(
\begin{array}{llllllllllllll}
a_1    \ldots    \ldots     \ldots     a_{\lambda_1}     \\
a_1     \ldots  \ldots               a_{\lambda_2} \\
 \ldots  \ldots   \\
a_1 \ldots a_{\lambda_p}
\end{array}
\right).
\end{equation}

The tableaux of  kind (\ref{Deruyts}) are called  {\it{  Deruyts  tableaux}}
(of shape $\lambda$)  in honor of Jacques Deruyts ($1862-1945$), who  
introduced
them in his treatment of \emph{semi-invariants/primary covariants} of algebraic forms
\cite{Deruyts} (see also \cite{Green1} and \cite{Bri-BR}).

Set
\begin{equation}\label{Coderuyts}
C_{\lambda} = \left(
\begin{array}{llllllllllllll}
a_1 \ldots    \ldots     \ldots     a_1                           \\
a_2   \ldots  \ldots               a_2 \\
 \ldots  \ldots   \\
a_p \ldots a_p
\end{array} \right).
\end{equation}

Since $C_{\lambda}$ is the conjugate tableau 
$C_{\lambda} = \widetilde{D_{\tilde{\lambda}}}$
of the Deruyts tableau $D_{\tilde{\lambda}}$ of shape $\tilde{\lambda}$,
we refer to the tableaux of kind (\ref{Coderuyts}) as  {\it{  Coderuyts  tableaux}}
(of shape $\lambda$).

Now, assume that the alphabet $\mathcal{A}$ is
$$
\mathcal{A} =   \mathcal{A}_0 \cup \mathcal{A}_1  \cup \mathcal{L}
$$
as in Section \ref{superalgebra first}.

Given a shape/partition $\lambda$, 
we assume that $|\mathcal{A}_0| = m_0 \geq \widetilde{\lambda}_1$ 
and $|\mathcal{A}_1| = m_1 \geq \lambda_1$.
Let us denote by $\alpha_1, \ldots, \alpha_p \in \mathcal{A}_0$, $\beta_1, \ldots, \beta_{\lambda_1} \in A_1$
two \emph{arbitrary} families of \emph{distinct positive and negative virtual symbols}, respectively.

Set
\begin{equation}\label{Deruyts and Coderuyts}
D_{\lambda}^* = \left(
\begin{array}{llllllllllllll}
\beta_1    \ldots    \ldots     \ldots     \beta_{\lambda_1}     \\
\beta_1     \ldots  \ldots               \beta_{\lambda_2} \\
 \ldots  \ldots   \\
\beta_1 \ldots \beta_{\lambda_p}
\end{array}
\right),
\qquad
C_{\lambda}^* = \left(
\begin{array}{llllllllllllll}
\alpha_1 \ldots    \ldots     \ldots     \alpha_1                           \\
\alpha_2   \ldots  \ldots               \alpha_2 \\
 \ldots  \ldots   \\
\alpha_p \ldots \alpha_p
\end{array} \right).
\end{equation}

The tableaux of  kind (\ref{Deruyts and Coderuyts}) are called \emph{virtual}  Deruyts and Coderuyts tableaux
of shape $\lambda$, respectively.

\subsubsection{Bitableaux monomials in ${\mathbf{U}}(gl(m_0+m_1,n))$}\label{Bitableaux monomials}
Let $S$ and $T$ be two Young tableaux of same shape $\lambda \vdash h$ on
the  alphabet $ \mathcal{A}_0 \cup \mathcal{A}_1 \cup \mathcal{L} $:

\begin{equation}\label{bitableaux}
S = \left(
\begin{array}{llllllllllllll}
z_{i_1}  \ldots    \ldots     \ldots     z_{i_{\lambda_1}}     \\
z_{j_1}   \ldots  \ldots               z_{j_{\lambda_2}} \\
 \ldots  \ldots   \\
z_{s_1} \ldots z_{s_{\lambda_p}}
\end{array}
\right), \qquad
T = \left(
\begin{array}{llllllllllllll}
z_{h_1}  \ldots    \ldots     \ldots     z_{h_{\lambda_1}}    \\
z_{k_1}   \ldots  \ldots               z_{k_{\lambda_2}} \\
 \ldots  \ldots   \\
z_{t_1} \ldots z_{t_{\lambda_p}}
\end{array}
\right).
\end{equation}

To the pair $(S,T)$, we associate the {\it{bitableau monomial}}:
\begin{equation}\label{BitMon}
e_{S,T} =
e_{z_{i_1}, z_{h_1}}\cdots e_{z_{i_{\lambda_1}}, z_{h_{\lambda_1}}}
e_{z_{j_1}, z_{k_1}}\cdots e_{z_{j_{\lambda_2}}, z_{k_{\lambda_2}}}
 \cdots  \cdots
e_{z_{s_1}, z_{t_1}}\cdots e_{z_{s_{\lambda_p}}, z_{t_{\lambda_p}}}
\end{equation}
in ${\mathbf{U}}(gl(m_0|m_1+n)).$

\subsubsection{Capelli bitableaux in ${\mathbf{U}}(gl(n))$}\label{Capbit sub}

Given a pair of Young tableaux  $S, T$ of the same shape $\lambda$ on the proper alphabet $\mathcal{L}$, consider the elements
$$
e_{S,C_{\lambda}^*} \ e_{C_{\lambda}^*,T} \in {\mathbf{U}}(gl(m_0|m_1+n)).
$$

Since these elements  are \emph{balanced monomials} in
${\mathbf{U}}(gl(m_0|m_1+n))$, then they belong to the \emph{virtual subalgebra} $Virt(m_0+m_1,n)$.

Hence, we can consider their images in ${\mathbf{U}}(gl(n))$ with respect to the Capelli epimorphism $\mathfrak{p}$.

We set
\begin{equation}\label{determinantal}
 \mathfrak{p} \Big( e_{S,C_{\lambda}^*} \ e_{C_{\lambda}^*,T}    \Big) = [S|T]    \in {\mathbf{U}}(gl(n)),
\end{equation}
and call the element $[S|T]$ a {\textit{Capelli bitableau}}.

The elements defined in (\ref{determinantal}) do not
depend on the choice of the virtual  Coderuyts tableau 
$C_{\lambda}^*$.

From \cite{BT-CA}, we recall that the Capelli bitableaux $[S|T] \in {\mathbf{U}}(gl(n))$
are the preimages of the \emph{determinantal bitableaux} $(S|T)$ in the polynomial algebra ${\mathbb C}[M_{n,n}]$
(see section \ref{standard} below) with respect to the \emph{Koszul equivariant isomorhism}
$$\mathcal{K} : \mathbf{U}(gl(n)) \rightarrow {\mathbb C}[M_{n,n}] \cong \mathbf{Sym}(gl(n)).$$

Hence, Capelli bitableaux $[S|T]$ admit explicit Laplace expansions: 
let $S$ and $T$ be the Young tableaux
$$
S = \left(
\begin{array}{llllllllllllll}
i_{p_1}  \ldots    \ldots     \ldots     i_{p_{\lambda_1}} \\
i_{q_1}   \ldots  \ldots               i_{q_{\lambda_2}} \\
 \ldots  \ldots   \\
i_{r_1} \ldots i_{r_{\lambda_m}}
\end{array}
\right),
\quad
T = \left(
\begin{array}{llllllllllllll}
j_{s_1}  \ldots    \ldots     \ldots     j_{s_{\lambda_1}} \\
j_{t_1}   \ldots  \ldots               j_{t_{\lambda_2}} \\
 \ldots  \ldots   \\
j_{v_1} \ldots j_{v_{\lambda_m}}
\end{array}
\right).
$$

\begin{proposition}\emph{( \cite{BT-CA}, Corollary $8.3$ )} We have
$$
[S|T] =
\sum_{\sigma_1, \ldots, \sigma_m } \ (-1)^{\sum_{k=1}^m \ |\sigma_k|} \
\left[
\begin{array}{c}
i_{p_{\sigma_1(1)}}\\   . \\ i_{p_{\sigma_1(\lambda_1)}} \\
\vdots   \\
i_{r_{\sigma_m(1)}}\\   . \\ i_{r_{\sigma_m(\lambda_m)}}
\end{array}
\right| \left.
\begin{array}{c}
j_{s_1}\\   . \\ j_{s_{\lambda_1}}  \\
\vdots \\
j_{v_1}\\   . \\ j_{v_{\lambda_m}}
\end{array}
\right]
$$
$$
\phantom{[S|T]} =
\sum_{\sigma_1, \ldots, \sigma_m } \ (-1)^{\sum_{k=1}^m \ |\sigma_k|} \
\left[
\begin{array}{c}
i_{p_1}\\   . \\ i_{p_{\lambda_1}}  \\
\vdots \\
i_{r_1}\\   . \\ i_{r_{\lambda_m}}
\end{array}
\right| \left.
\begin{array}{c}
j_{s_{\sigma_1(1)}}\\   . \\ j_{s_{\sigma_1(\lambda_1)}} \\
\vdots \\
j_{v_{\sigma_m(1)}}\\   . \\ j_{v_{\sigma_m(\lambda_m)}}
\end{array}
\right],
$$
where
$$
\left[
\begin{array}{c}
i_{p_{\sigma_1(1)}}\\   . \\ i_{p_{\sigma_1(\lambda_1)}} \\
\vdots   \\
i_{r_{\sigma_m(1)}}\\   . \\ i_{r_{\sigma_m(\lambda_m)}}
\end{array}
\right| \left.
\begin{array}{c}
j_{s_1}\\   . \\ j_{s_{\lambda_1}}  \\
\vdots \\
j_{v_1}\\   . \\ j_{v_{\lambda_m}}
\end{array}
\right],
\qquad
\left[
\begin{array}{c}
i_{p_1}\\   . \\ i_{p_{\lambda_1}}  \\
\vdots \\
i_{r_1}\\   . \\ i_{r_{\lambda_m}}
\end{array}
\right| \left.
\begin{array}{c}
j_{s_{\sigma_1(1)}}\\   . \\ j_{s_{\sigma_1(\lambda_1)}} \\
\vdots \\
j_{v_{\sigma_m(1)}}\\   . \\ j_{v_{\sigma_m(\lambda_m)}}
\end{array}
\right]
$$
are \emph{column Capelli bitableaux} in $\mathbf{U}(gl(n))$ \emph{(see e.g. \cite{BT-CA}, \cite{BT-jaa})}.
\end{proposition}

\subsubsection{Young-Capelli bitableaux in ${\mathbf{U}}(gl(n))$}

Given a pair of Young tableaux  $S, T$ of the same shape $\lambda$ on the proper alphabet $\mathcal{L}$, consider the elements
$$
e_{S,C_{\lambda}^*} \ e_{C_{\lambda}^*,D_{\lambda}^*}  \   e_{D_{\lambda}^*,T} \in {\mathbf{U}}(gl(m_0|m_1+n)).
$$

Since these elements  are \emph{balanced monomials} in
${\mathbf{U}}(gl(m_0|m_1+n))$, then they belong to the \emph{virtual subalgebra} $Virt(m_0+m_1,n)$.

Hence, we can consider their images in ${\mathbf{U}}(gl(n))$ with respect to the Capelli epimorphism $\mathfrak{p}$.

We set
\begin{equation}\label{rightYoungCapelli}
\mathfrak{p} \Big( e_{S,C_{\lambda}^*} \ e_{C_{\lambda}^*,D_{\lambda}^*}  \   e_{D_{\lambda}^*,T} \Big)
 =    [S| \fbox{$T$}]         \in {\mathbf{U}}(gl(n)).
\end{equation}
and call the element $[S| \fbox{$T$}] $ a {\textit{Young-Capelli bitableau}}.

The elements defined in  (\ref{rightYoungCapelli}) do not
depend on the choice of the virtual Deruyts and Coderuyts tableaux $D_{\lambda}^*$
and $C_{\lambda}^*$.

In plain words, the Young-Capelli bitableau $[S| \fbox{$T$}]$ is obtained 
from the Capelli bitableau $[S|T]$ by adding a \emph{column symmetrization} on the 
right Young tableau $T$. Indeed, we have

\begin{proposition}\label{YC as sym of Capelli} 
Any Young-Capelli bitableau
equals the sum of Capelli bitableaux:
$$
[S| \fbox{$T$}] =  \sum  [S|\overline{T}],
$$
where the sum is extended to all  Young tableaux $\overline{T}$ obtained from $T$ by  permutations of the elements
of each column.
\end{proposition} 

\subsubsection{Double Young-Capelli bitableaux in ${\mathbf{U}}(gl(n))$}

Given a pair of Young tableaux  $S, T$ of the same shape $\lambda$ on the proper alphabet $\mathcal{L}$
 consider the elements
 $$
e_{S,C_{\lambda}^{*}} \cdot 
e_{C_{\lambda}^{*},D_{\lambda}^{*}}
 \cdot  e_{D_{\lambda}^{*},C_{\lambda}^{*}}
\cdot e_{C_{\lambda}^{*},T} \in  {\mathbf{U}}(gl(m_0|m_1+n)).
$$

Since these elements  are \emph{balanced monomials} in
${\mathbf{U}}(gl(m_0|m_1+n))$, then they belong to the \emph{virtual subalgebra} $Virt(m_0+m_1,n)$.

Hence, we can consider their images in ${\mathbf{U}}(gl(n))$ with respect to the Capelli epimorphism $\mathfrak{p}$.

We set
\begin{equation}\label{doubleYoungCapelli}
\mathfrak{p} \Big( e_{S,C_{\lambda}^{*}} \cdot 
e_{C_{\lambda}^{*},D_{\lambda}^{*}}
 \cdot  e_{D_{\lambda}^{*},C_{\lambda}^{*}}
\cdot e_{C_{\lambda}^{*},T} \Big)
 =   [\ \fbox{$S \ | \ T$}\ ]         \in {\mathbf{U}}(gl(n)).
\end{equation}
and call the element $[\ \fbox{$S \ | \ T$}\ ] $ a {\textit{double Young-Capelli bitableau}}.

The elements defined in (\ref{doubleYoungCapelli}) do not
depend on the choice of the virtual Deruyts and Coderuyts tableaux $D_{\lambda}^*$
and $C_{\lambda}^*$.

In plain words, the double Young-Capelli bitableau
 $[\ \fbox{$S \ | \ T$}\ ]$
is obtained from the Young-Capelli bitableau $[S| \fbox{$T$}]$
by adding a further \emph{row skew symmetrization}. Indeed, we have

\begin{proposition}\label{double YC from YC}
Any double Young-Capelli bitableau
equals the sum of Young-Capelli bitableaux:
$$
[\ \fbox{$S \ | \ T$}\ ] = (-1)^ {h \choose 2} \ \sum_{\sigma} \ (-1)^{|\sigma|} \  [S|\fbox{$T^{\sigma}$}],
$$
where the sum is extended to all  Young tableaux $T^{\sigma}$ obtained from $T$ by  permutations of the elements
of each row, and $(-1)^{|\sigma|}$ is the product of the signatures of row permutations.
\end{proposition}

\subsection{Bitableaux in ${\mathbb C}[M_{m_0|m_1+n,d}]$ and the standard monomial theory}\label{standard}

\subsubsection{Biproducts in ${\mathbb C}[M_{m_0|m_1+n,d}]$}

Embed the algebra
$$
{\mathbb C}[M_{m_0|m_1+n,d}] = {\mathbb C}[(\alpha_s|j), (\beta_t|j), (i|j)]
$$
into the (supersymmetric) algebra ${\mathbb C}[(\alpha_s|j), (\beta_t|j), (i|j), (\gamma|j)]$
generated by the ($\mathbb{Z}_2$-graded) variables $(\alpha_s|j), (\beta_t|j), (i|j), (\gamma|j)$,
$j = 1, 2, \ldots, d$,
 where
 $$
 |(\gamma|j)| = 1 \in \mathbb{Z}_2 \ \  for \ every \ j = 1, 2, \ldots, d,
 $$
and denote by $D^l_{z_i,\gamma}$ the superpolarization of $\gamma$ to  $z_i.$

Let $\omega = z_1z_2 \cdots z_p$ be a word on     $ \mathcal{A}_0 \cup \mathcal{A}_cup \mathcal{L} $, 
and $\varpi = j_{t_1}j_{t_2} \cdots j_{t_q}$ a word
on the alphabet
$P = \{1, 2, \ldots, d \}$. The {\it{biproduct}}
$$
(\omega|\varpi) = (z_1z_2 \cdots z_p|j_{t_1}j_{t_2} \cdots j_{t_q})
$$
is the element
$$
D^l_{z_1,\gamma}D_{z_2,\gamma} \cdots D^l_{z_p,\gamma} \Big( (\gamma|j_{t_1})(\gamma|j_{t_2}) \cdots
(\gamma|j_{t_q}) \Big) \in {\mathbb C}[M_{m_0|m_1+n,d}]
$$
if $p = q$ and is set to be zero otherwise.

\begin{claim} 
The biproduct $(\omega|\varpi) = (z_1z_2 \cdots z_p|j_{t_1}j_{t_2} \cdots j_{t_q})$ is supersymmetric in the  $z$'s and
skew-symmetric in the  $j$'s.
In symbols
\begin{enumerate}

\item
$
(z_1 z_2 \cdots z_i z_{i+1} \cdots z_p|j_{t_1} j_{t_2} \cdots j_{t_q}) =
\\ \null \hfill (-1)^{|z_i| |z_{i+1}|}
(z_1 z_2 \cdots z_{i+1} z_i \cdots z_p|j_{t_1} j_{t_2} \cdots j_{t_q})
$

\item
$
(z_1z_2 \cdots z_iz_{i+1} \cdots z_p|j_{t_1}j_{t_2} \cdots j_{t_i}j_{t_{i+1}} \cdots j_{t_q}) =
\\ \null \hfill
- (z_1z_2 \cdots z_iz_{i+1} \cdots z_p|j_{t_1} \cdots j_{t_{i+1}}j_{t_i} \cdots j_{t_q}).
$
\end{enumerate}

\end{claim}

\begin{proposition}{\textbf{(Laplace expansions)}}\label{Laplace expansions}
We have
\begin{enumerate}

\item
$
(\omega_1\omega_2|\varpi) = \Sigma_{(\varpi)} \ (-1)^{|\varpi_{(1)}| |\omega_2|} \ 
(\omega_1|\varpi_{(1)})(\omega_2|\varpi_{(2)}).
$

\item
$
(\omega|\varpi_1\varpi_2) = \Sigma_{(\omega)} \  (-1)^{|\varpi_1| |\omega_{(2)}|}  \ 
(\omega_{(1)}|\varpi_1)(\omega_{(2)}|\varpi_2.)
$

\end{enumerate}
where
$$
\bigtriangleup(\varpi) = \Sigma_{(\varpi)}  \ \varpi_{(1)} \otimes \varpi_{(2)}, \quad \bigtriangleup(\omega)
= \Sigma_{(\omega)} \ \omega_{(1)} \otimes \omega_{(2)}
$$
denote the coproducts in the Sweedler notation \emph{(see, e.g. \cite{Abe-BR})} of the elements $\varpi$ and $\omega$ in the
supersymmetric Hopf algebra of $W$
\emph{(see, e.g. \cite{Bri-BR})}   and in
the free exterior Hopf algebra generated by
$j = 1, 2, \ldots, d$, respectively.

\end{proposition}

\begin{example} Let $\omega = \alpha_1 \alpha_2 3$, $\varpi = 1 2 3$,
where $|(\alpha_1|j)| = |(\alpha_2|j)| = 1,$  $j = 1, 2, 3$ and $|(3|j)|  = 0, \ j = 1, 2, 3$.
Then
\begin{align*}
(\omega|\varpi) &= D^l_{ \alpha_1,\gamma}D^l_{ \alpha_2,\gamma}D^l_{ 3,\gamma}  
\big( (\gamma|1)(\gamma|2)(\gamma|3) \big) \\
&=   D_{ \alpha_1,\gamma}D_{ \alpha_2,\gamma} \Big( (3|1)(\gamma|2)(\gamma|3)
- (\gamma|1)(3|2)(\gamma|3)   +  (\gamma|1)(\gamma|2)(3|3) \Big) \\
&= D_{ \alpha_1,\gamma} \Big( (3|1)(\alpha_2|2)(\gamma|3) + (3|1)(\gamma|2)(\alpha_2|3)
- (\alpha_2|1)(3|2)(\gamma|3) \\
& \phantom{D_{ \alpha_1,\gamma} \Big( \quad}- (\gamma|1)(3|2)(\alpha_2|3)
+ (\alpha_2|1)(\gamma|2)(3|3) + (\gamma|1)(\alpha_2|2)(3|3) \Big) \\
&= (3|1)(\alpha_2|2)(\alpha_1|3) + (3|1)(\alpha_1|2)(\alpha_2|3)
- (\alpha_2|1)(3|2)(\alpha_1|3) \\
& \phantom{= } - (\alpha_1|1)(3|2)(\alpha_2|3)
+ (\alpha_2|1)(\alpha_1|2)(3|3) + (\alpha_1|1)(\alpha_2|2)(3|3).
\end{align*}

\noindent From Proposition \ref{Laplace expansions}.1, by setting $\varpi_1 = 1 2$, $\varpi_2 = 3$, it follows
$$
(\omega|\varpi) = (\alpha_1\alpha_2|12)(3|3) + (\alpha_13|12)(\alpha_2|3) + (\alpha_2 3|12)(\alpha_1|3).
$$

\noindent From Proposition \ref{Laplace expansions}.2, by setting $\omega_1 = \alpha_1  \alpha_2$, $\omega_2 = 3$, it follows
$$
(\omega|\varpi) = (\alpha_1\alpha_2|12)(3|3) - (\alpha_1\alpha_2|13)(3|2) + (\alpha_1\alpha_2|23)(3|1).
$$
\end{example}

\subsubsection{Biproducts in ${\mathbb C}[M_{n,d}]$}

Let $\omega = i_1i_2 \cdots i_p$, $\varpi = j_1j_1 \cdots j_p$ be  words on the negative alphabet 
$\mathcal{L} = \{1, 2, \ldots, n \}$ and on  the negative alphabet 
$\mathcal{P} = \{1, 2, \ldots, d \}$.

From Proposition \ref{Laplace expansions}, we infer

\begin{corollary}
The {\it{biproduct}} of the two words $\omega$ and $\varpi$
\begin{equation}\label{biproduct}
(\omega|\varpi) = (i_1i_2 \cdots i_p|j_1j_2 \cdots j_p)
\end{equation}
is the  \emph{signed minor}:
$$
(\omega|\varpi) = (-1)^{p \choose 2} \ 
det \Big( \ (i_r|j_s) \ \Big)_{r, s = 1, 2, \ldots, p} \in {\mathbb C}[M_{n,d}]. 
$$
\end{corollary}

\subsubsection{Biproducts and polarization operators}

Following the notation introduced in the previous sections, let
$$
Super[W] = Sym[W_0] \otimes \Lambda[W_1]
$$
denote the {\it{(super)symmetric}} algebra of the space
$$
W = W_0 \oplus W_1
$$ (see, e.g. \cite{Scheu-BR}).

By multilinearity, the algebra
$Super[W]$ is the same as the superalgebra  $Super[\mathcal{A}_0 \cup  \mathcal{A}_1 \cup \mathcal{L} ]$ generated by the "variables"
$$
 \alpha_1, \ldots, \alpha_{m_0}  \in  \mathcal{A}_0                                              , \quad \beta_1, \ldots, \beta_{m_1} \in A_1, \quad x_1, \ldots, x_n \in L,
$$
modulo the congruences
$$
z z' = (-1)^{|z| |z'|} z' z, \quad z, z' \in \mathcal{A}_0 \cup  \mathcal{A}_1  \cup \mathcal{L} .
$$
Let $d^l_{z, z'}$ denote the (left)polarization operator of $z'$ to $z$ on
$$
Super[W] = Super[\mathcal{A}_0 \cup  \mathcal{A}_1 \cup \mathcal{L} ],
$$
that is the unique superderivation of $\mathbb{Z}_2$-degree
$$
|z| + |z'| \in \mathbb{Z}_2
$$
such that
$$
d^l_{z, z'} (z'') = \delta_{z', z''} \cdot z,
$$
for every $z, z', z'' \in \mathcal{A}_0 \cup  \mathcal{A}_1  \cup \mathcal{L} .$

Clearly, the map
$$
e_{z, z'}  \rightarrow   d^l_{z, z'}
$$
is a Lie superalgebra map and, therefore, induces a structure of
$$gl(m_0|m_1+n)-module$$
on
$Super[\mathcal{A}_0 \cup  \mathcal{A}_1 \cup \mathcal{L} ] = Super[W].$

\begin{proposition}\label{polarization biproduct}

Let $\varpi = j_{t_1}j_{t_2} \cdots j_{t_q}$ be a word on $P = \{1, 2, \ldots, d \}$.
The map
$$
\Phi_\varpi : \omega \mapsto (\omega|\varpi),
$$
$\omega$ any word on $\mathcal{A}_0 \cup  \mathcal{A}_1 \cup \mathcal{L} $, uniquely defines $gl(m_0|m_1+n)-$equivariant linear operator
$$
\Phi_\varpi : Super[\mathcal{A}_0 \cup  \mathcal{A}_1 \cup \mathcal{L} ] \rightarrow {\mathbb C}[M_{m_0|m_1+n,d}],
$$
that is
\begin{equation}\label{polarization action}
\Phi_\varpi \big( e_{z, z'} \cdot \omega \big) =\Phi_\varpi \big( d^l_{z, z'}(\omega) \big) =
 D^l_{z, z'} \big( (\omega|\varpi) \big) =
 e_{z, z'} \cdot (\omega|\varpi),
\end{equation}
for every $z, z' \in \mathcal{A}_0 \cup  \mathcal{A}_1 \cup \mathcal{L} .$
\end{proposition}

With a slight abuse of notation, we will write (\ref{polarization action}) in the form
\begin{equation}\label{abuse}
D^l_{z, z'} \big( (\omega|\varpi) \big) = ( d^l_{z, z'}(\omega)|\varpi).
\end{equation}

\subsubsection{Bitableaux in ${\mathbb C}[M_{m_0|m_1+n,d}]$}

Let $S = (\omega_1, \omega_2, \ldots, \omega_p$ and $T = (\varpi_1, \varpi_2, \ldots, \varpi_p)$ be Young tableaux on
$\mathcal{A}_0                                               \cup  \mathcal{A}_1                        \cup \mathcal{L} $ and $P = \{1, 2, \ldots, d \}$ of shapes $\lambda$ and $\mu$, respectively.

If $\lambda = \mu$, the {\it{Young bitableau}} $(S|T)$ is the element of ${\mathbb C}[M_{m_0|m_1+n,d}]$ defined as follows:
$$
(S|T) =
\left(
\begin{array}{c}
\omega_1\\ \omega_2\\ \vdots\\ \omega_p
\end{array}
\right| \left.
\begin{array}{c}
\varpi_1\\ \varpi_2\\ \vdots\\  \varpi_p
\end{array}
\right)
= \pm \ (\omega_1)|\varpi_1)(\omega_2)|\varpi_2) \cdots (\omega_p)|\varpi_p),
$$
where
$$
\pm  = (-1)^{|\omega_2||\varpi_1|+|\omega_3|(|\varpi_1|+|\varpi_2|)+ \cdots +|\omega_p|(|\varpi_1|+|\varpi_2|+\cdots+|\varpi_{p-1}|)}.
$$

If $\lambda \neq \mu$, the {\it{Young bitableau}} $(S|T)$ is set to be zero.

\subsubsection{Bitableaux and polarization operators}

By naturally extending the slight abuse of notation (\ref{abuse}), the action of any polarization on bitableaux
can be explicitly described:

\begin{proposition}\label{action on tableaux} 
Let $z, z' \in \mathcal{A}_0 \cup  \mathcal{A}_1 \cup \mathcal{L} $,  and let
$S = (\omega_1, \ldots, \omega_p) $, $T =
(\varpi_1, \ldots, \varpi_p)$. We have the
following identity:
\begin{align*}
e_{z, z'} \cdot (S\,|\,T) \ & = 
\ D^l_{z, z'} \ \big( \left(
\begin{array}{c}
\omega_1\\ \omega_2\\ \vdots\\ \omega_p
\end{array}
\right| \left.
\begin{array}{c}
\varpi_1\\ \varpi_2\\ \vdots\\  \varpi_p
\end{array}
\right) \big) \\ &
= \ \sum_{s=1}^p  \
(-1)^{(|z| + |z'|)\epsilon_s}
\ \left(
\begin{array}{c}
\omega_1\\ \omega_2\\ \vdots\\  d^l_{z,
z'}(\omega_s)\\ \vdots \\ \omega_p
\end{array}
\right| \left.
\begin{array}{c}
\varpi_1\\ \varpi_2\\ \vdots\\
\vdots \\ \vdots\\ \varpi_p
\end{array}
\right),
\end{align*}
where
$$
\epsilon_1 = 1, \quad    \epsilon_s = |\omega_1| + \cdots + |\omega_{s-1}|, \quad s = 2,
\ldots, p.
$$
\end{proposition}

\begin{example} Let $\alpha_i \in \mathcal{A}_0$, $1,  2, 3, 4 \in L$, $|D_{\alpha_i, 2}| = 1$. Then
$$
e_{\alpha_i, 2} \cdot
\left(
\begin{array}{lll}
1 \ 3 \ 2\\
 2 \ 3 \\
4 \ 2
\end{array}
\right| \left.
\begin{array}{lll}
1 \ 2 \ 3 \\
2 \ 3 \\
3 \ 1
\end{array}
\right) =
D^l_{\alpha_i, 2} \ \big(
\left(
\begin{array}{lll}
1 \ 3 \ 2\\
 2 \ 3 \\
4 \ 2
\end{array}
\right| \left.
\begin{array}{lll}
1 \ 2 \ 3 \\
2 \ 3 \\
3 \ 1
\end{array}
\right) \big) =
$$
$$ =
\left(
\begin{array}{lll}
1 \ 3 \ \alpha_i\\
 2 \ 3\\
4 \ 2
\end{array}
\right| \left.
\begin{array}{lll}
1 \ 2 \ 3 \\
2 \ 3 \\
3 \ 1
\end{array}
\right) -
\left(
\begin{array}{lll}
1 \ 3 \ 2\\
 \alpha_i \ 3\\
4 \ 2
\end{array}
\right| \left.
\begin{array}{lll}
1 \ 2 \ 3 \\
2 \ 3 \\
3 \ 1
\end{array}
\right) +
\left(
\begin{array}{lll}
1 \ 3 \ 2\\
 2 \ 3\\
4 \ \alpha_i
\end{array}
\right| \left.
\begin{array}{lll}
1 \ 2 \ 3 \\
2 \ 3 \\
3 \ 1
\end{array}
\right).
$$

\end{example}

\subsubsection{The straightening algorithm and the standard basis theorem for ${\mathbb C}[M_{m_0|m_1+n,d}]$}

Consider the set of all bitableaux $(S|T) \in {\mathbb C}[M_{m_0|m_1+n,d}]$, where $sh(S) = sh(T) \vdash h$,
$h$ a given positive integer. In the following, let
denote by $\leq$ the partial order on this set defined  by the following two steps:
\begin{enumerate}

\item 
$(S|T) < (S'|T')$ whenever $sh(S)  <_l sh(S') $,

\item 
$(S|T) < (S'|T')$ whenever $sh(S) =  sh(S')$, $w(S) >_l w(S')$,  $w(T) >_l w(T')$,
\end{enumerate}
where the shapes and the row-words are compared in the lexicographic
order.

The next results are superalgebraic versions of classical, well-known results for the symmetric algebra ${\mathbb C}[M_{n,d}]$
(\cite{drs-BR}, \cite{DKR-BR}, \cite{DEP-BR}, for the general theory of standard monomials see, e.g. \cite{Procesi-BR}, Chapt. 13)
and of their skew-symmetric analogues  (\cite{DR-BR}, \cite{ABW-BR}).

\begin{theorem}(The straightening algorithm)\label{theorem: standard expansion of bitableaux} \emph{\cite{rota-BR}}

Let $(P|Q) \in {\mathbb C}[M_{m_0|m_1+n,d}]$.
Then $(P|Q)$ can be
written as a linear combination, with rational coefficients,
\begin{equation}\label{straightening}
(P|Q) = \sum_{S,T} c_{S,T}\ (S|T),
\end{equation}
of standard bitableaux $(S|T)$, where $(S|T) \geq (P|Q)$ and $c_S = c_P$, $c_T = c_Q$.
\end{theorem}
Since standard bitableaux are linearly independent in ${\mathbb C}[M_{m_0|m_1+n,d}]$, the expansion
(\ref{straightening}) is unique.

Following \cite{Berele1-BR}, \cite{Berele2-BR},  \cite{Brini1-BR},
a partition $\lambda$ satisfies the $(m_0,m_1+n)-${\it{hook condition}} (in symbols, $\lambda \in H(m_0,m_1+n)$)
if and only if $\lambda_{m_0+1} \leq m_1+n.$ We have:
\begin{lemma}
There exists a standard tableau on $\mathcal{A}_0   \cup  \mathcal{A}_1 \cup \mathcal{L} $ of shape $\lambda$ if
and only if  $\lambda \in H(m_0,m_1+n).$
\end{lemma}

Given a positive integer $h \in \mathbb{Z}^+$,
let ${\mathbb C}_h[M_{m_0|m_1+n,d}]$ denote the $h-$th homogeneous component of
${\mathbb C}[M_{m_0|m_1+n,d}].$

From Theorem \ref{theorem: standard expansion of bitableaux}, it follows

\begin{corollary}\label{theorem: standard basis}
(The  Standard basis theorem for
${\mathbb C}_h[M_{m_0|m_1+n,d}]$, \cite{rota-BR})

The following set is a basis of ${\mathbb C}_h[M_{m_0|m_1+n,d}]$:
$$
\{ (S|T) \ standard;\ sh(S) = sh(T) = \lambda \vdash h,   \lambda \in H(m_0,m_1+n), \lambda_1                        \leq d \ \}.
$$
\end{corollary}

\section{The Schur (covariant) modules and supermodules}\label{Schur}

\subsection{The Schur  ${\mathbf{U}}(gl(m_0|m_1+n))$-supermodules as
submodules of $\mathbb{C}[M_{m_0|m_1+n,d}]$}

Given $\lambda \in H(m_0,m_1+n)$, the {\it{Schur supermodule}} $Schur_\lambda(m_0,m_1+n)$ is
the subspace of ${\mathbb C}[M_{m_0|m_1+n,d}]$, $d \geq \lambda_1$, spanned by the set of all bitableaux
$(S|D^P_\lambda)$ of shape $\lambda,$
where $S$ is a Young tableau on the alphabet $\mathcal{A}_0  \cup  \mathcal{A}_1 \cup \mathcal{L} $, and
$D^P_\lambda$ is the {\it{Deruyts}} tableau on $P = \{1, 2, \ldots, d \}$
$$
D^P_\lambda =  \left(
\begin{array}{llllllllllllll}
1  \ 2 \ 3 \ldots     \ldots    & \lambda_1 \\                           \\
1   \ 2 \ 3 \ldots               \lambda_2 \\
 \ldots  \ldots  & \\
1 \ 2 \ 3 \ldots \lambda_p
\end{array}
\right), \quad p = \textit{l}(\lambda).
$$

From Corollary \ref{theorem: standard basis} and the {\textit{straightening algorithm}} (see, e.g. \cite{Bri-BR}), it follows

\begin{proposition}

The set
$$
\Big\{ (S|D^P_\lambda); \ S \ superstandard   \       \Big\}
$$
is a ${\mathbb C}$-linear basis of $Schur_\lambda(m_0,m_1+n)$.
\end{proposition}

Furthermore, we recall
\begin{proposition}\emph{(\cite{Brini1-BR}, \cite{Bri-BR})}
The submodule $Schur_\lambda(m_0,m_1+n)$ is an irreducible $\mathbf{U}(gl(m_0|m_1+n))$-submodule of
${\mathbb C}[M_{m_0|m_1+n,d}]$, with highest weight
$$
(\lambda_1, \ldots, \lambda_{m_0}; \ \widetilde{\lambda}_1-m_0, \widetilde{\lambda}_2-m_0, \ldots ).
$$
\end{proposition}

\subsection{The classical Schur ${\mathbf{U}}(gl(n))$-modules}\label{Schur modules}

Given $\lambda$ such that $\lambda_1 \leq n$, the {\it{Schur module}} $Schur_\lambda(n)$ is
the subspace of ${\mathbb C}[M_{n,d}]$, $d \geq \lambda_1$, spanned by the set of all bitableaux
$(X|D^P_\lambda)$ of shape $\lambda$  and $X$ is a Young tableau on the alphabet $ L.$

\begin{proposition}

The set
$$
\Big\{ (X|D^P_\lambda); X \ standard   \       \Big\}
$$
is a ${\mathbb C}$-linear basis of $Schur_\lambda(n)$.

Furthermore, $Schur_\lambda(n)$ is an irreducible $\mathbf{U}(gl(n))$-submodule of
${\mathbb C}[M_{n,d}]$, with highest weight $\widetilde{\lambda}.$
\end{proposition}

Let
$$
D_{\lambda} = \left(
\begin{array}{llllllllllllll}
1  \ 2 \ 3 \ldots     \ldots    & \lambda_1 \\                           \\
1   \ 2 \ 3 \ldots               \lambda_2 \\
 \ldots  \ldots  & \\
1 \ 2 \ 3 \ldots \lambda_p
\end{array}
\right)
$$
denote the (proper) Deruyts tableau on the alphabet $\mathcal{L} = \{1, 2, \ldots, n \}.$
The bitableau
$$
v_{\widetilde{\lambda}} = (D_\lambda|D^P_\lambda)
$$
is the ``canonical''  highest weight vector of the irreducible $gl(n)$-module $Schur_\lambda(n)$
with highest weight $\widetilde{\lambda}$.

\subsection{The classical Schur modules as $gl(n)$-submodules of Schur supermodules}\label{Lemmas}

Let $\lambda = (\lambda_1, \lambda_2, \ldots, \lambda_p)$ be a partition such that $\lambda_1                        \leq n$.

Consider a Schur supermodule
$$
Schur_\lambda(m_0,m_1+n)
$$
(clearly, $\lambda \in H(m_0,m_1+n)$, for every $m_0, m_1$).

The Schur module $Schur_\lambda(n)$ can be regarded as a $\mathbf{U}(gl(n))$-submodule of the
$\mathbf{U}(gl(m_0|m_1+n))$-supermodule $Schur_\lambda(m_0,m_1+n).$

Let $\mathfrak{p}$ be the Capelli epimorphism 
$$
\mathfrak{p} : Virt(m_0+m_1, n) \twoheadrightarrow
{\mathbf{U}}(gl(n)), \qquad  Ker(\mathfrak{p}) = \mathbf{Irr}.
$$

\begin{proposition}
The Schur module $Schur_\lambda(n)$ is  invariant (as a subspace of $Schur_\lambda(m_0,m_1+n)$) with respect to the action
of the subalgebra
$$
Virt(m_0+m_1, n)   \subset \mathbf{U}(gl(m_0|m_1+n)).
$$
\end{proposition}

From Proposition \ref{virtual action}, we infer

\begin{proposition}
For every element $\mathbf{M} \in Virt(m_0+m_1, n)$, the action of $\mathbf{M}$ on 
the Schur module $Schur_\lambda(n)$
is the same of the action of its image $\mathfrak{p}(\mathbf{M}) = \mathbf{m} \in {\mathbf{U}}(gl(n)).$

\end{proposition}

\subsection{The action of double Young-Capelli bitableaux on highest weight vectors of Schur modules}
\label{sect Double YC action}

Let's start with some lemmas.

In the following, given partitions $\lambda$, $\mu$ and their conjugates
$\widetilde{\lambda}$ and $\widetilde{\mu}$,
we assume that 
$$
m_0 \geq  \lambda_1, \  \mu_1, \qquad m_1, d \geq \widetilde{\lambda}_1, 
\ \widetilde{\mu}_1.
$$.

Let $v_{\mu} = (D_{\widetilde{\mu}}|D^P_{\widetilde{\mu}})$ be the ``canonical''  highest weight vector 
of weight $\mu$
of the irreducible $gl(n)$-module $Schur_{\widetilde{\mu}}$.

\begin{lemma}\label{vanishing virtual lemma}
We have
\begin{align}
&\textrm{If} \ |\widetilde{\mu}| < |\widetilde{\lambda}|, \  \textrm{then}  &   
e_{C_{\widetilde{\lambda}}^*,S}  \cdot (D_{\widetilde{\mu}}|D^P_{\widetilde{\mu}}) &= 0,
\quad \forall S\label{due}
\\
&\textrm{If} \ |\widetilde{\mu}| = |\widetilde{\lambda}|, \ \widetilde{\mu} \neq \widetilde{\lambda}, \ 
\textrm{then}  &  e_{D_{\widetilde{\lambda}}^*,C_{\widetilde{\lambda}}^*} 
e_{ C_{\widetilde{\lambda}}^*,S} \cdot (D_{\widetilde{\mu}}|D^P_{\widetilde{\mu}}) &= 0,
\quad \forall S.\label{tre}
\end{align}
\end{lemma}
The assertions of eqs.  (\ref{due}),   (\ref{tre}) are special cases of standard
elementary facts of the method
of virtual variables (see, e.g. \cite{Bri-BR}).
\begin{lemma}\label{Vanishing Lemma}
If ${\widetilde{\lambda}}  \nsubseteq {\widetilde{\mu}}$,  then 
$$
e_{D_{\widetilde{\lambda}}^*,C_{\widetilde{\lambda}}^*} e_{ C_{\widetilde{\lambda}}^*,S} 
\cdot (D_{\widetilde{\mu}}|D^P_{\widetilde{\mu}}) = 0,
\quad \forall S.
$$
\end{lemma}

\begin{proof}
Assume  that $|\widetilde{\mu}| \geq |\widetilde{\lambda}|$ to avoid trivial cases (by eq. (\ref{due})).
The action $e_{C_{\widetilde{\lambda}}^*,S} \cdot (D_{\widetilde{\mu}}|D^P_{\widetilde{\mu}})$ produces a linear combination of bitableaux
$(T|D^P_{\widetilde{\mu}}) \in Schur_{\widetilde{\mu}}(m_0,m_1+n)$, where each tableau $T$ contains
 exactly ${\widetilde{\lambda}}_i$ occurrences of the positive virtual
symbols $\alpha_i \in \mathcal{A}_0                                              $.
By  {\it{straightening}} each of them \emph{(see, e.g. \cite{Bri-BR})},
the element $e_{C_{\widetilde{\lambda}}^*,S} \cdot (D_{\widetilde{\mu}}|D^P_{\widetilde{\mu}}) $ is uniquely expressed as a linear combination of
(super)standard tableaux
\begin{equation}\label{dominance}
e_{C_{\widetilde{\lambda}}^*,S} \cdot (D_{\widetilde{\mu}}|D^P_{\widetilde{\mu}})  = \sum_i \ 
(S_i|D^P_{\widetilde{\mu}}) \in Schur_{\widetilde{\mu}}(m_0,m_1+n),
\end{equation}
where in each $S_i$ the positive virtual symbols $\alpha_i \in \mathcal{A}_0                                              $ occupies a subshape $\widetilde{\lambda}' \subseteq 
{\widetilde{\mu}}$ such that
$\widetilde{\lambda}'   \trianglerighteq \widetilde{\lambda}$.
If $\widetilde{\lambda} \nsubseteq \widetilde{\mu}$, any element $(S_i|D^P_{\widetilde{\lambda}})$ in the canonical form (\ref{dominance}) is such that $\widetilde{\lambda}'   \trianglerighteq \widetilde{\lambda}$,
$\widetilde{\lambda}'   \neq \widetilde{\lambda}$. Then $e_{D_{\widetilde{\lambda}}^*,C_{\widetilde{\lambda}}^*} \cdot (S_i|D^P_{\widetilde{\mu}})  = 0$, by skew-symmetry, and the assertion follows.
\end{proof}

We recall that, given a shape/partition $\lambda = (\lambda_1      \geq \lambda_2 \geq \cdots \geq \lambda _p)$,
the {\it{hook length}}  $H(x)$ of a box $x$ in the Ferrers diagram $F_\lambda$ of the shape $\lambda$
is the number of boxes that are in the same row to the
right of it plus those boxes in the same column below it, plus one (for the box itself).
The {\it{hook number}} the shape $\lambda$ is the product $H(\lambda) = \prod_{x \in F_\lambda} \ H(x).$
Furthermore, we write $\lambda!$ for the product $\lambda_1!\lambda_2!  \cdots  \lambda_p!$.

\begin{lemma}(Regonati's Hook Lemma, \emph{\cite{Regonati-BR}}) \label{hook lemma}
Let $H(\lambda) = H(\widetilde{\lambda})$ denotes the hook number of the shape/partition $\lambda.$
We have
\begin{align}
e_{C_{\widetilde{\lambda}}^{*},D_{\widetilde{\lambda}}} \cdot  v_{\lambda} &= 
e_{C_{\widetilde{\lambda}}^{*},D_{\widetilde{\lambda}}} 
\cdot (D_{\widetilde{\lambda}}|D^P_{\widetilde{\lambda}})
\\
&= (-1)^{{k} \choose   {2}} \frac {H({\widetilde{\lambda}})} {{\widetilde{\lambda}}!}
\ (C_{{\widetilde{\lambda}}}^{*}|D^P_{\widetilde{\lambda}}) \label{hook proper}
\end{align}
and
\begin{equation}
 e_{C_{{\widetilde{\lambda}}}^{*},D_{\widetilde{\lambda}}^*} 
\cdot (D_{\widetilde{\lambda}}^*|D^P_{\widetilde{\lambda}})
= (-1)^{{k} \choose   {2}} \frac {H({\widetilde{\lambda}})} {{\widetilde{\lambda}}!} \  
(C_{{\widetilde{\lambda}}}^{*}|D^P_{\widetilde{\lambda}}) \label{hook virtual}.
\end{equation}
Furthermore 
\begin{equation}
e_{D_{{\widetilde{\lambda}}}^{*},C_{\widetilde{\lambda}}^{*}} 
\cdot (C_{{\widetilde{\lambda}}}^{*}|D^P_{\widetilde{\lambda}})   = \ {\widetilde{\lambda}}!
\ (D_{{\widetilde{\lambda}}}^{*}|D^P_{\widetilde{\lambda}})\label{trivial}.
\end{equation}
\end{lemma}

\begin{theorem}\label{Double YC action} We have:
\begin{enumerate}

\item If  $sh(S) =  sh(S') = \widetilde{\lambda}$,  $|\widetilde{\mu}| < |\widetilde{\lambda}|$,  then  
$$   
[\ \fbox{$S' \ | \ S$}\ ]  \cdot v_\mu = 0,
$$
\item If $sh(S) =  sh(S') =\widetilde{\lambda}$, $|\widetilde{\mu}| = |\widetilde{\lambda}|$,
 $\widetilde{\mu} \neq \widetilde{\lambda}$, then  
$$   
[\ \fbox{$S' \ | \ S$}\ ]  \cdot v_\mu = 0,
$$
\item If  $\widetilde{\mu} = \widetilde{\lambda}$,   then
$$   
[\ \fbox{$D_{\widetilde{\lambda}} \ | \ D_{\widetilde{\lambda}}$}\ ]  \cdot v_\lambda 
= H(\widetilde{\lambda})^2 \
v_\lambda,
$$
\item If $sh(S) =  sh(S') = \widetilde{\lambda}$,  
$\widetilde{\lambda} \nsubseteq \widetilde{\mu}$,  
then
$$  
[\ \fbox{$S' \ | \ S$}\ ]   \cdot v_\mu = 0.
$$
\end{enumerate}
\end{theorem}
\begin{proof} Since
$$
[\ \fbox{$S' \ | \ S$}\ ]  = \mathfrak{p} \Big( e_{S',C_{\lambda}^{*}} \cdot 
e_{C_{\lambda}^{*},D_{\lambda}^{*}}
 \cdot  e_{D_{\lambda}^{*},C_{\lambda}^{*}}
\cdot e_{C_{\lambda}^{*},S} \Big) \in {\mathbf{U}}(gl(n)),
$$
then
$$
[\ \fbox{$S' \ | \ S$}\ ]  \cdot v_\mu = e_{S'C_{\lambda}^{*}} \cdot 
e_{C_{\lambda}^{*},D_{\lambda}^{*}}
 \cdot  e_{D_{\lambda}^{*},C_{\lambda}^{*}}
\cdot e_{C_{\lambda}^{*},S} \cdot v_\mu,
$$
from Proposition \ref{virtual action}.

Hence, item $1)$ follows from eq. (\ref{due}), item $2)$ follows 
from eq. (\ref{tre}) and item $4)$ follows from Lemma \ref{vanishing virtual lemma}.

Then, we prove item $3)$. We have

$$
[\ \fbox{$D_{\widetilde{\lambda}} \ | \ D_{\widetilde{\lambda}}$}\ ]  \cdot v_\lambda =
  e_{D_{\widetilde{\lambda}},C_{\widetilde{\lambda}}^{*}} \ 
e_{C_{\widetilde{\lambda}}^{*},D_{\widetilde{\lambda}}^{*}} 
\ e_{D_{\widetilde{\lambda}}^{*},C_{\widetilde{\lambda}}^{*}} \
 e_{C_{\widetilde{\lambda}}^{*},D_{\widetilde{\lambda}}} \cdot (D_{\widetilde{\lambda}}|D^P_{\widetilde{\lambda}}).\label{first term}
$$
From eq. (\ref{hook proper}),  this equals
$$
 \frac {1} {\widetilde{\lambda}! } (-1)^{{k} \choose {2}}  
  H(\lambda) \
 e_{D_{\lambda},C_{\lambda}^{*}} \ e_{C_{\lambda}^{*},D_{\lambda}^{*}} \ e_{D_{\lambda}^{*},C_{\lambda}^{*}} 
\cdot (C_{\widetilde{\lambda}}^{*}|D^P_{\widetilde{\lambda}});
$$ 
from eq. (\ref{trivial}) this equals
$$
 (-1)^{{k} \choose {2}}  H(\widetilde{\lambda}) \
 e_{D_{\widetilde{\lambda}},C_{\widetilde{\lambda}}^{*}} \ 
e_{C_{\widetilde{\lambda}}^{*},D_{\widetilde{\lambda}}^{*}}  \
\cdot (D_{\widetilde{\lambda}}^{*}|D^P_{\widetilde{\lambda}});
$$
from eq. (\ref{hook virtual}) this equals
\begin{align*}
&= \ (-1)^{{k} \choose {2}}    
H(\widetilde{\lambda}) \frac {1} {\widetilde{\lambda}!} (-1)^{{k} \choose {2}} 
\ e_{D_{\widetilde{\lambda}},C_{\widetilde{\lambda}}^{*}} \cdot 
(C_{\widetilde{\lambda}}^{*}|D^P_{\widetilde{\lambda}}) 
\\
&= H(\widetilde{\lambda})^2 \ (D_{\widetilde{\lambda}}|D^P_{\widetilde{\lambda}}) 
= H(\widetilde{\lambda})^2 	\ v_\lambda.
\end{align*}

\end{proof}

\section{The center $\boldsymbol{\zeta}(n)$ of $\mathbf{U}(gl(n))$}\label{The center}

\subsection{The  Schur elements  $\mathbf{S}_\lambda(n) \in \boldsymbol{\zeta}(n)$ }\label{Schur elements}

Let $\mu$ be a partition,  let $\tilde{\mu}$ be its conjugate partition. Assume $\mu_1 \leq n$,
and  $\ m_1 \geq \mu_1$, $m_0 \geq \widetilde{\mu}_1$; hence, the virtual Deruyts tableau $D^*_{\mu}$ and 
the virtual Coderuyts tableau $C^*_{\mu}$ can be constructed.
Let $S_1, S_2$ be tableaux  on the proper alphabet $\mathcal{L} = \{1, 2, \ldots, n \}$ of shape $\mu$.
We notice that any element
$$
e_{S_1,C_{\mu}^{*}} \cdot e_{C_{\mu}^{*},D_{\mu}^{*}} \cdot  e_{D_{\mu}^{*},C_{\mu}^{*}}
\cdot e_{C_{\mu}^{*},S_2} \in Virt(m_0+m_1,n),
$$
is {\textit{skew-symmetric}}  in the rows of $S_1$ and  $S_2$, respectively.

 Given a partition $\lambda$, assume $\tilde{\lambda}_1 \leq n$, $m_1 \geq \tilde{\lambda}_1$,
$ m_0 \geq \lambda_1$.
We set
\begin{align}
\mathbf{S}_{\lambda}(n) &= \frac {1} {H(\tilde{\lambda})} \ \  \sum_S \
 \mathfrak{p} \Big( e_{S, C_{  \tilde{\lambda}^{*}}^{*} } 
e_{C_{\tilde{\lambda}}^{*},D_{\tilde{\lambda}}^{*}}
\cdot  e_{D_{\tilde{\lambda}}^{*},C_{\tilde{\lambda}}^{*}}
\cdot e_{C_{\tilde{\lambda}}^{*},S} \big) \label{Schur basis1}
\\
&= \frac {1} {H(\tilde{\lambda})} \ \  \sum_S \ [\ \fbox{$S \ | \ S$}\ ]
\in {\mathbf{U}}(gl(n)), \label{Schur basis2}
\end{align}
where the sum is extended to all  row (strictly) increasing tableaux $S$ of shape $\tilde{\lambda}$ on the proper alphabet $\mathcal{L} = \{1, 2, \ldots, n \}$. Notice that $H(\tilde{\lambda}) = H(\lambda)$.

By convention, if $\lambda$ is the empty partition, we set $\mathbf{S}_{\emptyset}(n) = \mathbf{1} \in 
\boldsymbol{\zeta}(n).$

The element $\mathbf{S}_{\lambda}(n) \in {\mathbf{U}}(gl(n))$ is called the {\it{Schur element}} of 
{\textit{weight}} $\lambda$ (and 
shape $\tilde{\lambda}$) in dimension $n$.

\begin{theorem}
The Schur elements $\mathbf{S}_{\lambda}(n)$ are  central in $\mathbf{U}(gl(n))$.
\end{theorem}
\begin{proof}

Consider the element
$$
 \sum_S \  e_{S,C_{\tilde{\lambda}}^{*}} \cdot 
e_{C_{\tilde{\lambda}}^{*},D_{\tilde{\lambda}}^{*}}
 \cdot  e_{D_{\tilde{\lambda}}^{*},C_{\tilde{\lambda}}^{*}}
\cdot e_{C_{\tilde{\lambda}}^{*},S} \in Virt(m_0+m_1,n),
$$
where the sum is extended to all  row (strictly) increasing tableaux $S$ on the proper alphabet 
$\mathcal{L} = \{1, 2, \ldots,n \}$.

Since the adjoint representation acts by derivation, we have
$$
ad(e_{i j})\big( \sum_S \  e_{S,C_{\tilde{\lambda}}^{*}} \cdot 
e_{C_{\tilde{\lambda}}^{*},D_{\tilde{\lambda}}^{*}}
 \cdot  e_{D_{\tilde{\lambda}}^{*},C_{\tilde{\lambda}}^{*}}
\cdot e_{C_{\tilde{\lambda}}^{*},S}   \big) = 0,
$$
for every  $e_{i j} \in gl(n).$
Hence, the assertion follows from
Corollary \ref{centrality gen}.
\end{proof}

Let $\boldsymbol{\zeta}(n)^{(m)}$ denote the $m$-th filtration element of $\boldsymbol{\zeta}(n)$ with respect to the filtration induced by
the standard filtration of $\mathbf{U}(gl(n)).$
Clearly,
\begin{equation}\label{filtration element}
\mathbf{S}_{\lambda}(n) \in \boldsymbol{\zeta}(n)^{(m)},
\end{equation}
for every $m \geq |\lambda|.$

\begin{theorem}(Triangularity/orthogonality of the actions on highest weight vectors)\label{Schur action}
We have:
\begin{align}
&\textrm{If} \ |\mu| < |\lambda|, \  \textrm{then}  &   \mathbf{S}_{\lambda}(n) \cdot v_{\mu} &= 0,\label{Sc 1}
\\
&\textrm{If} \ |\mu| = |\lambda|, \ \textrm{then} &
\mathbf{S}_{\lambda} (n) \cdot v_{\mu} &= \delta_{\lambda, \mu} \ H(\lambda) \ 
v_{\lambda}.\label{Sc 2}
\end{align}

\end{theorem}
\begin{proof}
The first assertion is an immediate consequence of Theorem \ref{Double YC action}, item $1)$.
The fact that, if $|\mu| = |\lambda|,$ $\mu \neq \lambda,$ then $\mathbf{S}_{\lambda} (n) \cdot v_{\mu} = 0,$
is an immediate consequence of Theorem \ref{Double YC action}, item $2)$.

We examine the case $\lambda = \mu$.
The value
$$
\mathbf{S}_{\lambda}(n) \cdot v_{\lambda}  \stackrel{def}{=} \frac {1} {H(\widetilde{\lambda})} \
 \sum_S \
\mathfrak{p} \big(e_{S,C_{\widetilde{\lambda}}^{*}} \ e_{C_{\widetilde{\lambda}}^{*},D_{\widetilde{\lambda}}^{*}} 
\ e_{D_{\widetilde{\lambda}}^{*},C_{\widetilde{\lambda}}^{*}}  
e_{C_{\widetilde{\lambda}}^{*},S}\big) \cdot (D_{\widetilde{\lambda}}|D^P_{\widetilde{\lambda}})
$$
equals 
$$\label{virtual eq}
\frac {1} {H(\widetilde{\lambda})} \
 \sum_S \
e_{S,C_{\widetilde{\lambda}}^{*}} \ e_{C_{\widetilde{\lambda}}^{*},D_{\widetilde{\lambda}}^{*}} 
\ e_{D_{\widetilde{\lambda}}^{*},C_{\widetilde{\lambda}}^{*}}  
e_{C_{\widetilde{\lambda}}^{*},S} \cdot (D_{\widetilde{\lambda}}|D^P_{\widetilde{\lambda}}),
$$
by Proposition \ref{virtual action}.
Clearly, this reduces to
$$
\frac {1} {H(\widetilde{\lambda})} \  e_{D_{\widetilde{\lambda}},C_{\widetilde{\lambda}}^{*}} \ 
e_{C_{\widetilde{\lambda}}^{*},D_{\widetilde{\lambda}}^{*}} 
\ e_{D_{\widetilde{\lambda}}^{*},C_{\widetilde{\lambda}}^{*}} \
 e_{C_{\widetilde{\lambda}}^{*},D_{\widetilde{\lambda}}} \cdot (D_{\widetilde{\lambda}}|D^P_{\widetilde{\lambda}}).\label{first term}
$$
This value equals
$$
\frac {1} {H(\widetilde{\lambda})} \
 [\ \fbox{$D_{\widetilde{\lambda}} \ | \ D_{\widetilde{\lambda}}$}\ ]  \cdot v_\lambda 
= H(\widetilde{\lambda}) \ v_\lambda,
$$
by item $3)$ of Theorem \ref{Double YC action}.
\end{proof}

\begin{theorem}(Vanishing theorem)\label{Vanishing theorem}
If  $\lambda   \nsubseteq \mu$ , then   
$\mathbf{S}_{\lambda}(n) \cdot v_{\mu} = 0.\label{v 1}
$
\end{theorem}
\begin{proof} It is an immediate consequence of item $4)$ of Theorem \ref{Double YC action}.
\end{proof}

\begin{theorem}\label{Schur basis}

For every $m \in \mathbb{Z}^+$, the set
$$
\big\{ \mathbf{S}_{\lambda}(n); \ \tilde{\lambda}_1  \leq n, \ |\lambda| \leq m \ \big\}
$$
is a linear basis of $\boldsymbol{\zeta}(n)^{(m)}.$

The set
$$
\big\{ \mathbf{S}_{\lambda}(n); \ \tilde{\lambda}_1 \leq n \ \big\}
$$
is a linear basis of the center $\boldsymbol{\zeta}(n).$

\end{theorem}

\subsection{The Sahi-Okounkov Characterization  Theorem}\label{Characterization  Theorems}
We reword Theorem \ref{Schur action} in terms of the \textit{Harish-Chandra
isomorphism} 
$$
\chi_n : \boldsymbol{\zeta}(n) \longrightarrow \Lambda^*(n),
$$
where $\Lambda^*(n)$ denotes the algebra of {\it{shifted symmetric polynomials}} in $n$ variables 
(see Section \ref{Lambda(n)} below).

\begin{proposition}\label{Char Schur shift}  Given  $\lambda$, $\tilde{\lambda}_1 \leq n$ and 
$\mu$, $\tilde{\mu}_1 \leq n$, we have:
\begin{align*}
&\textrm{If} \ |\mu| < |\lambda|, \  \textrm{then}  &   
\chi_n\left( \mathbf{S}_{\lambda}(n) \right) (\mu) &= 0,
\\
&\textrm{If} \ |\mu| = |\lambda|, \ \textrm{then} &
\chi_n\left( \mathbf{S}_{\lambda}(n) \right) (\mu) &= \delta_{\lambda, \mu} \ H(\lambda).
\end{align*}

\end{proposition}

We recall the {\it{Sahi-Okounkov Characterization Theorem}} for the 
\emph{Schur shifted symmetric polynomials}
$$
s^*_{\lambda|n}, \qquad \lambda_1 \leq n
$$
(Theorem $1$ of \cite{Sahi2-BR} and Theorem $3.3$ of \cite{OkOlsh-BR}, see also \cite{Okounkov-BR}).

\begin{proposition}\label{SahiOkounkv}  Given $\lambda$, $\widetilde{\lambda}_1 \leq n$, 
the polynomial $s^*_{\lambda|n}$
is the unique element of $\Lambda^*(n)$
such that $deg \ s^*_{\lambda|n} \leq |\lambda|$
and
$$
s^*_{\lambda|n}(\mu) = \delta_{\lambda \mu} \ H(\lambda),
$$
for all partitions $\mu$ such that $|\mu| \leq |\lambda|$ and $\mu_1 \leq n$.
\end{proposition}

From Propositions \ref{Char Schur shift} and \ref{SahiOkounkv},  
we obtain
\begin{corollary} Given  $\lambda$, $\tilde{\lambda}_1 \leq n$, we have
$$
\chi_n\left( \mathbf{S}_{\lambda}(n) \right) = s_{\lambda|n}^*.
$$ 
\end{corollary}

It follows that the Schur elements $\mathbf{S}_{\lambda}(n)$
are the same as the {\it{Okounkov quantum immanant}} associated to $\lambda$ 
(\cite{Okounkov-BR}, see also \cite{Okounkov1-BR} and \cite{Nazarov2-BR}).

We recall that the Schur element/Okounkov quantum immanant $\mathbf{S}_{\lambda}(n)$ 
also admits a representation as linear  combination of  \emph{Capelli immanants} 
(see \cite{BT-jaa}, section 5)
\begin{align*}
Cimm_{\mu}[i_1 i_2 \cdots i_h;j_1 j_2 \cdots j_h] & =
\sum_{\sigma \in \mathbf{S}_h} \ \chi^{\mu}(\sigma) \left[
\begin{array}{c}
i_{\sigma(1)}\\  i_{\sigma(2)} \\ \vdots \\ i_{\sigma(h)}
\end{array}
\right| \left.
\begin{array}{c}
j_1\\  j_2 \\ \vdots \\ j_h
\end{array}
\right]
\\
& =\sum_{\sigma \in \mathbf{S}_h} \ \chi^{\mu}(\sigma)  \left[
\begin{array}{c}
i_1\\  i_2 \\ \vdots \\ i_h
\end{array}
\right| \left.
\begin{array}{c}
j_{\sigma(1)}\\  j_{\sigma(2)} \\ \vdots \\ j_{\sigma(h)}
\end{array}
\right],
\end{align*}
where $\chi^{\mu}$ denotes the {\textit{irreducible character}} associated
to the irreducible representation of shape $\mu$ of the symmetric group $\mathbf{S}_h$ and
$$
\left[
\begin{array}{c}
i_{\sigma(1)}\\  i_{\sigma(2)} \\ \vdots \\ i_{\sigma(h)}
\end{array}
\right| \left.
\begin{array}{c}
j_1\\  j_2 \\ \vdots \\ j_h
\end{array}
\right], 
\quad
\left[
\begin{array}{c}
i_1\\  i_2 \\ \vdots \\ i_h
\end{array}
\right| \left.
\begin{array}{c}
j_{\sigma(1)}\\  j_{\sigma(2)} \\ \vdots \\ j_{\sigma(h)}
\end{array}
\right]
$$
are \textit{column Capelli bitableaux} (see, e.g. \cite{BT-jaa}, \cite{BT-CA}):

\begin{proposition} \emph{( \cite{BT-jaa}, Theorem $6.2$ )} Given $\lambda$, $\lambda_1 \leq n$, we have
$$
\mathbf{S}_{\lambda}(n) = (-1)^{h \choose 2}
\sum_{h_1 + \cdots + h_n = h} \ \frac {1} {h_1!  \cdots h_n!}
 \ Cimm_{\tilde{\lambda}}[1^{h_1} \ldots n^{h_n}; 1^{h_1} \ldots n^{h_n}].
$$

\end{proposition}

Hence, we have the remarkable identity:
\begin{corollary}\label{beautiful identity}
\begin{align}
\mathbf{S}_{\lambda}(n) &= \frac {1} {H(\tilde{\lambda})} \ \  \sum_S \ [\ \fbox{$S \ | \ S$}\ ]
\label{schur first eq}
\\
&= (-1)^{h \choose 2}
\sum_{h_1 + \cdots + h_n = h} \ \frac {1} {h_1!  \cdots h_n!}
 \ Cimm_{\tilde{\lambda}}[1^{h_1} \ldots n^{h_n}; 1^{h_1} \ldots n^{h_n}] \label{schur second eq},
\end{align}
where the sum of \emph{double Young-Capelli bitableaux} in eq. (\ref{schur first eq})  is extended to all  
row (strictly) increasing tableaux $S$ of shape $\tilde{\lambda}$ on the proper alphabet $\mathcal{L} = \{1, 2, \ldots, n \}$ and  eq. (\ref{schur second eq})
is a sum of \emph{diagonal Capelli immanants} of shape $\tilde{\lambda}$.
\end{corollary}

Corollary \ref{beautiful identity} was announced, without proof, in our recent paper \cite{BT-jaa}.

\subsection{The determinantal  Capelli generators $\mathbf{H}_k(n)$}

Let $(k)$ be the row shape of length $k$,  $\alpha \in \mathcal{A}_0 $  a positive virtual symbol.
The element
$$
[ i_1  i_2 \cdots i_k | j_1 j_2 \cdots j_k ] = 
\mathfrak{p} \big( e_{i_1 , \alpha}  e_{i_2, \alpha}  \cdots e_{i_k, \alpha}
e_{\alpha, j_1}e_{\alpha, j_2} \cdots e_{\alpha, j_k } \big)
$$
is the Capelli bitableau
$$
\mathfrak{p} \big( e_{S,C_{(k)}^*} \ e_{C_{(k)}^*,T} \big) \in {\mathbf{U}}(gl(n)),
$$
where $S = (i_1i_2 \cdots i_k)$ and  $T = (j_1j_2 \cdots j_k)$.
Clearly,  the elements $[ i_1  i_2 \cdots i_k | j_1 j_2 \cdots j_k ]$ are
skew-symmetric both in the left and the right sequences. In particular,
$$
[ i_k \cdots i_2 i_1 | i_1 i_2 \cdots i_k ] = (-1)^{k \choose 2} [ i_1 i_2 \cdots i_k| i_1 i_2 \cdots i_k ].
$$

In the enveloping algebra $\mathbf{U}(gl(n))$, given any integer $k = 1, 2, \ldots, n,$ consider 
the \textit{Capelli elements}
\begin{equation}\label{Capelli elements virtual}
\mathbf{H}_k(n)  =
 \sum_{1 \leq i_1 < \cdots < i_k \leq n} \ [ i_k \cdots i_2 i_1 | i_1 i_2 \cdots i_k ]. 
 \end{equation}

We recall that the Capelli elements admit a classical  presentation as a column determinant
\footnote{ The {\textit{column determinant}}
of a matrix $A = [a_{ij}]$ with noncommutative entries is, by definition,
$\mathbf{cdet} (A) = \sum_{\sigma}  (-1){|\sigma|} \  \ a_{\sigma(1), 1}a_{\sigma(2), 2} \cdots a_{\sigma(n), n}$.}
\cite{Cap1-BR}.
\begin{proposition}\label{Capelli classical} For every $k = 1, 2, \ldots, n,$, we have:
\begin{align*}
\mathbf{H}_k(n) & =
 \sum_{1 \leq i_1 < \cdots < i_k \leq n} \ [ i_k \cdots i_2 i_1 | i_1 i_2 \cdots i_k ]
\\
& = \sum_{1 \leq i_1 < \cdots < i_k \leq n}  \mathbf{cdet}\left(
 \begin{array}{cccc}
 e_{{i_1},{i_1}}+(k-1)   & e_{{i_1},{i_2}}       & \ldots    & e_{{i_1},{i_k}} \\
 e_{{i_2},{i_1}}         & e_{{i_2},{i_2}}+(k-2) & \ldots    & e_{{i_2},{i_k}} \\
 \vdots                  &    \vdots             & \vdots    & \vdots          \\
 e_{{i_k},{i_1}}         & e_{{i_k},{i_2}}       & \ldots    & e_{{i_k},{i_k}}
 \end{array}
 \right). 
\end{align*}
\end{proposition}
\begin{proof} See \cite{BT-CA}, Proposition $8.6$ (see also \cite{Koszul-BR}).
\end{proof}

\begin{proposition}
Let   $(1)^k $ be the column shape of 
depth $k$. Then,
$$
 \mathbf{H}_k(n)  = \mathbf{S}_{(1^k)}(n) \in \boldsymbol{\zeta}(n).
$$
\end{proposition}
\begin{proof}

We have
\begin{align*}
\mathbf{S}_{(1^k)}(n) &= \frac {1} {H((k))} \  \ \sum_S \ \mathfrak{p} 
\big(e_{S,C_{(k)}^{*}} \ e_{C_{(k)}^{*},D_{(k)}^{*}} \  e_{D_{(k)}^{*},C_{(k)}^{*}} \ e_{C_{(k)}^{*},S} \big) \\
&,
\end{align*}
where the sum is extended to all strictly increasing row tableaux $S$ of shape $(k)$ and
$H((k)) = k!$.

Notice that
$$
\mathfrak{p} \big(e_{S,C_{(k)}^{*}} \cdot e_{C_{(k)}^{*},D_{(k)}^{*}} \cdot  
e_{D_{(k)}^{*},C_{(k)}^{*}} \cdot e_{C_{(k)}^{*},S} \big)
$$
equals
$$
(-1)^{{k} \choose {2}} \
\mathfrak{p} \big(e_{S,C_{(k)}^{*}} \cdot e_{C_{(k)}^{*},C_{(k)}^{*}}  \cdot e_{C_{(k)}^{*},S} \big),
$$
that, in turn, equals
$$
(-1)^{{k} \choose {2}} \ k! \
\mathfrak{p} \big(e_{S,C_{(k)}^{*}}   \cdot e_{C_{(k)}^{*},S} \big) = 
(-1)^{{k} \choose {2}} \ k! \ \mathfrak{p} \big( e_{i_1 , \alpha}  e_{i_2, \alpha} \cdots e_{i_k, \alpha}
e_{\alpha, i_1}e_{\alpha, i_2} \cdots e_{\alpha, i_k } \big).
$$
Hence,
\begin{align*}
\mathbf{S}_{(1^k)}(n) &= 
(-1)^{{k} \choose {2}} \ \sum_{1 \leq i_1 < \cdots < i_k \leq n} \ \mathfrak{p} \big( e_{i_1 , \alpha}  e_{i_2, \alpha} \cdots e_{i_k, \alpha}
e_{\alpha, i_1}e_{\alpha, i_2} \cdots e_{\alpha, i_k } \big)
\\
&=\sum_{1 \leq i_1 < \cdots < i_k \leq n} \ \mathfrak{p} \big( e_{i_k , \alpha} \cdots e_{i_2, \alpha} e_{i_1, \alpha}
e_{\alpha, i_1}e_{\alpha, i_2} \cdots e_{\alpha, i_k } \big)
\\
&= \sum_{1 \leq i_1 < \cdots < i_k \leq n} \ [ i_k \cdots i_2 i_1 | i_1 i_2 \cdots i_k ] = \mathbf{H}_k(n).
\end{align*}
\end{proof}

\begin{corollary} The Capelli elements $\mathbf{H}_k(n)$ are central.
Furthermore,
$$
\mathbf{H}_k(n) \in \boldsymbol{\zeta}(n)^{(m)},
$$
for every $m \geq k.$
\end{corollary}

We recall the following fundamental result, indeed proved by  Capelli in  two  papers 
(\cite{Cap2-BR}, \cite{Cap3-BR}) with deceiving titles.

\begin{theorem}(Capelli, 1893)\label{Capelli generators}

The set
$$
\mathbf{H}_1(n), \mathbf{H}_2(n), \ldots, \mathbf{H}_n(n)
$$
is a set of algebraically independent generators of the center $\boldsymbol{\zeta}(n)$ of
$\mathbf{U}(gl(n)).$
\end{theorem}

As usual in the theory of symmetric functions, given a shape
$$\lambda = (\lambda_1       \geq \cdots \geq \lambda_p), \ \lambda_1     \leq n,$$
we set
$$
\mathbf{H}_{\lambda}(n) = \mathbf{H}_{\lambda_1}(n)\mathbf{H}_{\lambda_2}(n) \cdots \mathbf{H}_{\lambda_p}(n).
$$
By convention, if $\lambda$ is the empty partition, we set $\mathbf{H}_{\emptyset}(n) = \mathbf{1} \in \boldsymbol{\zeta}(n).$

From Theorem \ref{Capelli generators}, one infers

\begin{corollary}

The set
$$
\big\{ \mathbf{H}_{\lambda}(n); \lambda_1 \leq n, \ |\lambda| \leq m \ \big\}
$$
is a linear basis of $\boldsymbol{\zeta}(n)^{(m)}.$
\end{corollary}

We recall that  $v_{\widetilde{\mu}} = (D_\mu|D^P_\mu)$ denotes the  ``canonical''
highest weight vector of the Schur module $Schur_{\mu}(n)$, $\mu_1 \leq n$,  which is indeed of weight
$\widetilde{\mu}$ (Subsection \ref{Schur modules}).

Furthermore, we will write $\mathbf{H}_k(n) \cdot v_{\widetilde{\mu}}$ to mean the action of the central element
$\mathbf{H}_k(n)$ on $v_{\widetilde{\mu}}.$

For every $k = 1, 2, \ldots, n$, let
\begin{equation}\label{elementary shift}
e^*_k(\widetilde{\mu})
= \sum_{1 \leq i_1 < i_2 < \cdots < i_k \leq n} \ (\widetilde{\mu}_{i_1}  + k  - 1)
(\widetilde{\mu}_{i_2}  + k - 2) \cdots (\widetilde{\mu}_{i_k}).
\end{equation}

\begin{remark}
In formula (\ref{elementary shift}),
the sum can be regarded
as extended over all Ferrers subdiagrams obtained from the
Ferrers diagram of the partition $\mu$ by selecting $k$ columns $i_1 < i_2 < \cdots < i_k$,
and each summand
$$
(\widetilde{\mu}_{i_1}  + k  - 1)(\widetilde{\mu}_{i_2}  + k - 2) \cdots (\widetilde{\mu}_{i_k})
$$
is the product of the hook length $H(x)$ of the boxes of the first row
of each Ferrers subdiagram.
\end{remark} \qed

The classical determinatal presentation (Proposition \ref{Capelli classical})  of the $\mathbf{H}_k(n)$'s implies the following result.

\begin{proposition}\label{Capelli eigenvalues}
We have
$$
\mathbf{H}_k(n) \cdot v_{\widetilde{\mu}} = e^*_k(\widetilde{\mu})\cdot v_{\widetilde{\mu}}, 
\quad e^*_k(\widetilde{\mu}) \in \mathbb{Z}.
$$
\end{proposition}

\begin{corollary} If $\mu_1 < k$, then
$$
\mathbf{H}_k(n) \cdot v_{\widetilde{\mu}} = 0,
$$
\end{corollary}

The ``{\it{virtual definition}}'' (\ref{Capelli elements virtual}) of the $\mathbf{H}_k(n)$'s leads to a further combinatorial description of
the integer eigenvalues $e^*_k(\widetilde{\mu})$, which will turn out to be crucial
in the section on {\it{duality}}.

\begin{proposition}\label{horizontal strip} We have

$$
e^*_k(\widetilde{\mu}) = \sum \ hstrip_{\mu}(k)!,
$$
where the sum is extended to all ``horizontal strips''
\footnote{In this work, we use the expression
{\it{horizontal strip}} in a generalized sense. To wit, a  horizontal strip in a Ferrers diagram is a subset
of cells such that no two cells in the subset appear in the same column.}
of length $k$ in the Ferrers diagram of the partition $\mu$,
and the symbol $\ hstrip_{\mu}(k)!$ denotes the products of the factorials of the cardinality of each ``horizontal
component''
\footnote{For each each generalized horizontal strip, a {\it{horizontal
component}} is the set all cells on the same row.}
of the horizontal strip.
\end{proposition}

\begin{proof} Let
\begin{align*}
\mathbf{H}_k(n) &= \sum_{1 \leq i_1 < \cdots < i_k \leq n} \ [ i_k \cdots i_2 i_1 | i_1 i_2 \cdots i_k ]
\\
&=\sum_{1 \leq i_1 < \cdots < i_k \leq n} \ \mathfrak{p} \big( e_{i_k , \alpha} \cdots e_{i_2, \alpha} 
e_{i_1, \alpha} e_{\alpha, i_1}e_{\alpha, i_2} \cdots e_{\alpha, i_k } \big).
\end{align*}
Let $v_{\widetilde{\mu}} = (D_\mu|D^P_\mu)$ be the canonical  highest weight vector of the irreducible $gl(n)$-module $Schur_\mu(n)$ (of weight $\widetilde{\mu}$.)

Recall that
\begin{multline*}
\mathfrak{p} \big( e_{i_k , \alpha} \cdots e_{i_2, \alpha} 
e_{i_1, \alpha} e_{\alpha, i_1}e_{\alpha, i_2} \cdots e_{\alpha, i_k } \big) \cdot (D_\mu|D^P_\mu) =
\\ =  e_{i_k , \alpha} \cdots e_{i_2, \alpha} 
e_{i_1, \alpha} e_{\alpha, i_1}e_{\alpha, i_2} \cdots e_{\alpha, i_k }  \cdot (D_\mu|D^P_\mu),
\end{multline*}
by Proposition \ref{virtual action}.

The action of each summand of the ``virtualizing part''
$$
 e_{\alpha, i_1}e_{\alpha, i_2} \cdots e_{\alpha, i_k }
$$
distributes the $k$ occurrences of $\alpha$  in all horizontal strips of length $k$ (with column positions
$i_1, i_2, \ldots, i_k$) in the Ferrers diagram
of the partition $\mu,$ with signs - according to Remark \ref{action on tableaux} - since
$|e_{\alpha, i_h }| = 1$.
By applying   the ``devirtualizing part''
$$
e_{i_k , \alpha} \cdots e_{i_2, \alpha} e_{i_1, \alpha}
$$
it is easy to see that, for each horizontal strip, we obtain a sum of tableaux that:
\begin{itemize}
\item [--] to be non zero,  have the occurrences of $\alpha$ -  in any horizontal component of the strip - replaced
by a permutation
of the elements that  have been previously polarized into $\alpha$,
\item [--] have a sign that is easily seen to be
the product of the signs of the permutations of the elements in each horizontal component.
\end{itemize}
By reordering  each horizontal component, all the signs cancel out. Therefore,
we get
the  ``canonical''
highest weight vector $v_{\widetilde{\mu}} = (D_\mu|D^P_\mu)$ with
a positive
integer coefficient that is the product of the factorials of the lengths of the horizontal components. 
\end{proof}

\subsection{The  permanental Nazarov generators $\mathbf{I}_k(n)$}

In this section we provide the virtual  form of the set of the preimages in $\boldsymbol{\zeta}(n)$
- with respect to the Harish-Chandra isomorphism - of the sequence of {\textit{shifted complete  symmetric polynomials}}
$\mathbf{h}_k^{*}(x_1, x_2, \ldots, x_n),$  $k \in \mathbb{Z}^+$ (see \cite{OkOlsh-BR} and \cite{Molev1-BR}, Theorem $4.9$).

The central elements $\mathbf{I}_k(n),$ $k \in \mathbb{Z}^+,$ coincide (see \cite{BriUMI-BR}) with the ``permanental generators'' of
$\boldsymbol{\zeta}(n)$ originally discovered and studied - through the  machinery of {\it{Yangians}} - by Nazarov \cite{Nazarov-BR}
and later described by Umeda \cite{UmedaHirai-BR} as sums of {\textit{column permanents}}\footnote{ The {\textit{column permanent}}
of a matrix $A = [a_{ij}]$ with noncommutative entries is, by definition,
$\mathbf{cper} (A) = \sum_{\sigma}    \ a_{\sigma(1), 1}a_{\sigma(2), 2} \cdots a_{\sigma(n), n}$.}
in $\mathbf{U}(gl(n))$
(see e.g. Example \ref{cper} below, see also \cite{MolevNazarov-BR}, \cite{Nazarov2-BR}, and 
Turnbull  \cite{TURN-BR}).

The element
$$
[n^{h_n} \cdots 2^{h_2} 1^{h_1} | 1^{h_1} 2^{h_2} \cdots n^{h_n}]^{*} =
\mathfrak{p} \big( e_{n , \beta}^{h_n} \cdots e_{2, \beta}^{h_2} e_{1, \beta}^{h_1}
e_{\beta, 1}^{h_1}e_{\beta, 2}^{h_2} \cdots e_{\beta, n}^{h_n} \big),
$$
where $\beta \in A_1$ denotes {\it{any}} negative virtual symbol,
is symmetric both in the left and the right sequences.
In particular,
$$
[n^{h_n} \cdots 2^{h_2} 1^{h_1} | 1^{h_1} 2^{h_2} \cdots n^{h_n} ]^{*}
=
[1^{h_1} 2^{h_2} \cdots n^{h_n} | 1^{h_1} 2^{h_2} \cdots n^{h_n} ]^{*}.
$$

\begin{remark}\label{monbit and *-rows}
 Let $k = h_1 + h_2 + \cdots + h_n$, and let $(1)^k$ be the column shape
of depth $k$. Since
\begin{equation}\label{*-row}
e_{n , \beta}^{h_n} \cdots e_{2, \beta}^{h_2} e_{1, \beta}^{h_1}
e_{\beta, 1}^{h_1}e_{\beta, 2}^{h_2} \cdots e_{\beta, n}^{h_n} =
e_{1 , \beta}^{h_1}  e_{2, \beta}^{h_2} \cdots e_{n, \beta}^{h_n}
e_{\beta, 1}^{h_1}e_{\beta, 2}^{h_2} \cdots e_{\beta, n}^{h_n}
\end{equation}
the element \emph{(\ref{*-row})} equals the \emph{bitableau monomial} (see formula \emph{(\ref{BitMon})})
$$
e_{T, D^*_{(1)^k}} e_{D^*_{(1)^k}, T} \in {\mathbf{U}}(gl(m_0|m_1+n)), 
$$
where $T$ is the column tableau of shape $(1)^k$ with 
$\widetilde{T} = (1^{h_1} 2^{h_2} \cdots n^{h_n})$,
$sh(\widetilde{T}) = (k)$.
\end{remark}

In the enveloping algebra $\mathbf{U}(gl(n))$, given any positive integer $k \in \mathbb{Z}^+$, consider 
the \textit{Nazarov elements}
\begin{equation} \label{multiset}
\mathbf{I}_k(n)  =
 \sum_{(h_1, h_2, \ldots, h_n) } \ (h_1! h_2! \cdots h_n!)^{-1} \ 
[n^{h_n} \cdots 2^{h_2} 1^{h_1} | 1^{h_1} 2^{h_2} \cdots n^{h_n}]^{*},
\end{equation}
where the sum is extended to all $n$-tuples $(h_1, h_2, \ldots, h_n)$ such that $h_1 + h_2 + \cdots + h_n = k.$
Clearly, formula (\ref{multiset}) can be rewritten as
\begin{multline}\label{words}
\sum_{\underline{i} = (1 \leq i_1 \leq  \cdots \leq i_k \leq n)}
\ \big(h_1(\underline{i})!  \cdots h_n(\underline{i})! \big)^{-1} \
\mathfrak{p} \big( e_{i_k , \beta} \cdots e_{i_1, \beta}
e_{\beta, i_1} \cdots e_{\beta, i_k} \big),
\end{multline}
where, given a non decreasing $k$-tuple $\underline{i} = (1 \leq i_1 \leq  \cdots \leq i_k \leq n)$, 
we set
$$
h_j(\underline{i}) = \sharp \{ i_s = j; s = 1,  \ldots , k \}, \quad j = 1, 2, \ldots, n.
$$

In ``nonvirtual form'', the summands
$$
[n^{h_n} \cdots 2^{h_2} 1^{h_1} | 1^{h_1} 2^{h_2} \cdots n^{h_n} ]^{*}
$$
can be written as column permanent in the algebra $\mathbf{U}(gl(n))$ (see, e.g. \cite{UmedaHirai-BR}).

\begin{example}\label{cper}
\begin{multline*}
\mathbf{I}_3(3) = \frac{1}{3!}[111|111]^*+\frac{1}{2!}[211|112]^*+\frac{1}{2!}[311|113]^*+\frac{1}{2!}[221|122]^*
+[321|123]^*+
\\
+\frac{1}{2!}[331|133]^*+\frac{1}{3!}[222|222]^*+\frac{1}{2!}[322|223]^*+\frac{1}{2!}[332|233]^*+
\frac{1}{3!}[333|333]^*=
\\
=\frac{1}{3!}\mathbf{cper}
\left(
 \begin{array}{ccc}
 e_{1,1} - 2 & e_{1,1} - 1 &  e_{1,1} \\
 e_{1,1} - 2 & e_{1,1} - 1 &  e_{1,1}\\
 e_{1,1} - 2 & e_{1,1} - 1 &  e_{1,1}\\
 \end{array}
 \right)
+
\frac{1}{2!}\mathbf{cper}
\left(
 \begin{array}{ccc}
 e_{1,1} - 2 & e_{1,1} - 1 &  e_{1,2} \\
 e_{1,1} - 2 & e_{1,1} - 1 &  e_{1,2}\\
 e_{2,1}  & e_{2,1}  &  e_{2,2}\\
 \end{array}
 \right)+
\\
+\frac{1}{2!}\mathbf{cper}
\left(
 \begin{array}{ccc}
 e_{1,1} - 2 & e_{1,1} - 1 &  e_{1,3} \\
 e_{1,1} - 2 & e_{1,1} - 1 &  e_{1,3}\\
 e_{3,1}  & e_{3,1}  &  e_{3,3}\\
 \end{array}
 \right)
+
\frac{1}{2!}\mathbf{cper}
\left(
 \begin{array}{ccc}
 e_{1,1} - 2 & e_{1,2}  &  e_{1,2} \\
 e_{2,1}  & e_{2,2} - 1 &  e_{2,2}\\
 e_{2,1}  & e_{2,2} - 1 &  e_{2,2}\\
 \end{array}
 \right)+
\\
+\mathbf{cper}
\left(
 \begin{array}{ccc}
 e_{1,1} - 2 & e_{1,2} &  e_{1,3} \\
 e_{2,1} & e_{2,2} - 1 &  e_{2,3}\\
 e_{3,1} & e_{3,2} &  e_{3,3}\\
 \end{array}
 \right)
+
\frac{1}{2!}\mathbf{cper}
\left(
 \begin{array}{ccc}
 e_{1,1} - 2 & e_{1,3}  &  e_{1,3} \\
 e_{3,1}  & e_{3,3} - 1 &  e_{3,3}\\
 e_{3,1}  & e_{3,3} - 1 &  e_{3,3}\\
 \end{array}
 \right)+
\\
+\frac{1}{3!}\mathbf{cper}
\left(
 \begin{array}{ccc}
 e_{2,2} - 2 & e_{2,2} - 1 &  e_{2,2} \\
 e_{2,2} - 2 & e_{2,2} - 1 &  e_{2,2}\\
 e_{2,2} - 2 & e_{2,2} - 1 &  e_{2,2}\\
 \end{array}
 \right)
+
\frac{1}{2!}\mathbf{cper}
\left(
 \begin{array}{ccc}
 e_{2,2} - 2 & e_{2,2} - 1 &  e_{2,3} \\
 e_{2,2} - 2 & e_{2,2} - 1 &  e_{2,3}\\
 e_{3,2}  & e_{3,2}  &  e_{3,3}\\
 \end{array}
 \right)+
\\
+\frac{1}{2!}\mathbf{cper}
\left(
 \begin{array}{ccc}
 e_{2,2} - 2 & e_{2,3}  &  e_{2,3} \\
 e_{3,2}  & e_{3,3} - 1 &  e_{3,3}\\
 e_{3,2}  & e_{3,3} - 1 &  e_{3,3}\\
 \end{array}
 \right)
+
\frac{1}{3!}\mathbf{cper}
\left(
 \begin{array}{ccc}
 e_{3,3} - 2 & e_{3,3} - 1 &  e_{3,3} \\
 e_{3,3} - 2 & e_{3,3} - 1 &  e_{3,3}\\
 e_{3,3} - 2 & e_{3,3} - 1 &  e_{3,3}\\
 \end{array}
 \right).
\end{multline*}\qed

\end{example}

\begin{proposition}
Let $ (k)$ be the row shape of length $k.$ Then,
$$
\mathbf{I}_k(n) = \mathbf{S}_{(k)}(n).
$$
\end{proposition}
\begin{proof}
By formula (\ref{multiset}), we have
$$
\mathbf{S}_{(k)}(n) = \frac {1} {H((1)^k)} \  \ \sum_S \ \mathfrak{p} \big(e_{S,C_{(1^k)}^{*}} \ 
e_{C_{(1^k)}^{*},D_{(1^k)}^{*}} \  e_{D_{(1^k)}^{*},C_{(1^k)}^{*}} \ e_{C_{(1^k)}^{*},S} \big),
$$
where the sum is extended to all column tableaux $S$ of shape $(1^k)$
and $H((1)^k) = k!$.

Since  $S$ is a column tableaux  of shape $(1^k)$ and the column tableau $C_{(1^k)}^{*}$ is 
$$
C_{(1^k)}^{*} =
\left(
\begin{array}{c}
\alpha_1\\  \alpha_2 \\ \vdots \\ \alpha_k
\end{array}
\right),
$$
where the $\alpha_i$'s are \textit{distinct positive} virtual symbols, then each summand
$$
\mathfrak{p} \big(e_{S,C_{(1^k)}^{*}} \cdot e_{C_{(1^k)}^{*},D_{(1^k)}^{*}} \cdot  e_{D_{(1^k)}^{*},C_{(1^k)}^{*}} \cdot e_{C_{(1^k)}^{*},S} \big)
$$
equals
$$
\mathfrak{p} \big(e_{S,D_{(1^k)}^{*}}  \cdot e_{D_{(1^k)}^{*},S} \big).
$$
Hence,
\begin{align*}
\mathbf{S}_{(k)}(n) &= \frac {1} {k!} \sum_S \ \mathfrak{p} \big( e_{S,D_{(1^k)}^{*}}  e_{D_{(1^k)}^{*},S} \big)
\\
&= \frac {1} {k!} \sum_{(h_1, \ldots, h_n)} \ \sum_T \  \mathfrak{p} \big( e_{T,D_{(1^k)}^{*}}  e_{D_{(1^k)}^{*},T} \big),
\end{align*}
where the outer sum is extended over all indexes $h_1 + \cdots + h_n = k$ and inner sum
is extended over all column tableaux $T$ with $h_1$ occurrences of $1$, $h_2$ occurrences of $2$,
$\ldots$, $h_n$ occurrences of $n.$
Moreover, since each element  $e_{T,D_{(1^k)}^{*}}$ and  $e_{D_{(1^k)}^{*},T}$ are commutative, then
the inner sum
$$
\sum_T \ e_{T,D_{(1^k)}^{*}}  e_{D_{(1^k)}^{*},T}
$$
equals
$$
{{k} \choose  {h_1,h_2, \ldots, h_n}}
\ e_{1,\beta_1}^{h_1} \cdots e_{n,\beta_1}^{h_n} \ e_{\beta_1,1}^{h_1} \cdots e_{\beta_n,1}^{h_n},
$$
where there are  $h_1$ occurrences of $1$, $h_2$ occurrences of $2$,
$\ldots$, $h_n$ occurrences of $n.$
Hence, from Remark \ref{monbit and *-rows}, we infer
\begin{align*}
\mathbf{S}_{(k)}(n) &= \frac {1} {k!} \sum_{(h_1, \ldots, h_n)} \ \sum_T \ \mathfrak{p} 
\big( e_{T,D_{(1^k)}^{*}}  e_{D_{(1^k)}^{*},T} \big)
\\
&= \frac {1} {k!} \sum_{(h_1, \ldots, h_n)} \ {{k} \choose  {h_1,h_2, \ldots, h_n}}
\ \mathfrak{p} \big( e_{1,\beta_1}^{h_1} \cdots e_{n,\beta_1}^{h_n} \ e_{\beta_1,1}^{h_1} \cdots e_{\beta_n,1}^{h_n} \big)
\\
&= \sum_{(h_1, h_2, \ldots, h_n) } \ \frac {1} {h_1! h_2! \cdots h_n!} \ [n^{h_n} \cdots 2^{h_2} 1^{h_1} | 1^{h_1} 2^{h_2} \cdots n^{h_n} ]^{*} = \mathbf{I}_k(n).
\end{align*}
\end{proof}

\begin{corollary} The Nazarov elements $\mathbf{I}_k(n)$ are central.
Furthermore,
$$
\mathbf{I}_k(n) \in \boldsymbol{\zeta}(n)^{(m)},
$$
for every $m \geq k.$
\end{corollary}

The next characterization of the eigenvalues $h^*_k(\widetilde{\mu})$ of the elements
$\mathbf{I}_k(n)$, in combination with characterization of the eigenvalues $e^*_k(\widetilde{\mu})$ of the elements
$\mathbf{H}_k(n)$ (see Proposition \ref{horizontal strip}), will play a crucial role in our treatment of
{\textit{duality}} in the center $\boldsymbol{\zeta}(n)$ (see Section \ref{central duality} below).

\begin{theorem}\label{vertical strip}

We have:
$$
\mathbf{I}_k(n) \cdot v_{\widetilde{\mu}} = h^*_k(\widetilde{\mu})\cdot v_{\widetilde{\mu}}, \quad h^*_k(\widetilde{\mu}) \in \mathbb{N}
$$
with
$$
h^*_k(\widetilde{\mu}) = \sum \ vstrip_{\mu}(k)!,
$$
where the sum is extended to all ``vertical strips'' \footnote{In this work, we use the expression
{\it{vertical strip}} in a generalized sense. To wit, a  vertical strip in a Ferrers diagram is a subset
of cells such that no two cells in the subset appear in the same row.}
of length $k$ in the Ferrers diagram of the partition $\mu$,
and the symbol $\ vstrip_{\mu}(k)!$ denotes the product of the factorials of the cardinality of each vertical
component of the vertical strip.
\end{theorem}
\begin{proof}
The action of the ``virtualizing part''
$$
e_{\beta, 1}^{h_1}e_{\beta, 2}^{h_2} \cdots e_{\beta, n}^{h_n}
$$
of each summand in expression (\ref{multiset}) distributes
$k$ occurrences of the virtual variable $\beta$  in the Ferrers diagram
of the shape $\mu$, with $h_1$ occurrences in column $1$, $h_2$ occurrences in column $2$,
and so on.
Since $|\beta| = 1$, in order to get a non zero result,
these $\beta$'s must appear in different rows - by skew-symmetry - and, therefore, they form a
vertical strip. Clearly, this configuration is created $h_1! h_2! \cdots h_n!$ times.
Again  by skew-symmetry, the action of the ``devirtualizing part''
$$
e_{n , \beta}^{h_n} \cdots e_{2, \beta}^{h_2} e_{1, \beta}^{h_1}
$$
gives a non zero result if and only if the $\beta$'s in column $1$ are replaced by $1$,
the $\beta$'s in column $2$ are replaced by $2$, and so on. Therefore we obtain again
the highest weight vector $v_{\widetilde{\mu}} = (D_{\mu}|D_{\mu}^P)$ with multiplicity
$h_1! h_2! \cdots h_n!$.
Note that, since $|e_{\beta, p}| = |e_{p , \beta}| = 0$, for every $p = 1, 2, \ldots, n$,
no signs are involved in the proof.
\end{proof}

\begin{corollary}
If \ $\widetilde{\mu}_1 < k$, then
$$
\mathbf{I}_k(n) \cdot v_{\widetilde{\mu}} = 0.
$$
\end{corollary}

The eigenvalue $h^*_k(\widetilde{\mu})$ admits a further description that relates it to
{\textit{complete homogeneous shifted symmetric polynomials}}.

\begin{theorem}\label{complete}
We have:

\begin{equation}\label{complete shift}
h^*_k(\widetilde{\mu})
= \sum_{1 \leq i_1 \leq i_2 < \cdots \leq i_k \leq n} \ (\widetilde{\mu}_{i_1}  - k  + 1)
(\widetilde{\mu}_{i_2}  - k + 2) \cdots (\widetilde{\mu}_{i_{k-1}} - 1)\widetilde{\mu}_{i_k}
\end{equation}
\end{theorem}

\begin{proof}
The action of the ``virtualizing part''
$$
e_{\beta, i_1} e_{\beta, i_2} \cdots e_{\beta, i_k}, \quad \underline{i} = (i_1 \leq i_2 \leq \cdots   \leq  i_k),
$$
of each summand in expression (\ref{words}) of $\mathbf{I}_k(n)$, distributes
{\textit{one}} occurrence of the virtual variable $\beta$  in the Ferrers diagram
of the shape $\mu$, in column $i_k$, ...,  $i_2$,  $i_1$.
Since $|\beta| = 1$, in order to get a non zero result,
these $\beta$'s must be distributed into different rows, by skew-symmetry.
Clearly, this procedure can be done into
$$
(\widetilde{\mu}_{i_1}  - k  + 1)
(\widetilde{\mu}_{i_2}  - k + 2) \cdots (\widetilde{\mu}_{i_{k-1}} - 1)\widetilde{\mu}_{i_k}
$$ ways.
Again by by skew-symmetry, the action of the ``devirtualizing part''
$$
e_{i_k , \beta} e_{i_{k-1} , \beta} \cdots e_{i_1, \beta}
$$
gives a non zero result of and only if the $\beta$ in column $i_1$ is replaced by $i_1$,
the $\beta$ in column $i_2$ is replaced by $i_2$, and so on. Therefore we obtain again
the highest weight vector $v_{\widetilde{\mu}} = (D_{\mu}|D_{\mu}^P)$ with multiplicity
$h_1(\underline{i})!  \cdots h_n(\underline{i})!$.
\end{proof}

\subsection{Duality in $\boldsymbol{\zeta}(n)$}\label{central duality}

Let
$$
\mathcal{W}_n : \boldsymbol{\zeta}(n) \rightarrow \boldsymbol{\zeta}(n)
$$
be the algebra automorphism defined by setting
$$
\mathcal{W}_n \Big(  \mathbf{H}_k(n)  \Big) = \mathbf{I}_k(n), \quad k = 1, 2, \ldots, n.
$$

Clearly, Proposition \ref{horizontal strip} and Theorem \ref{vertical strip}.$2$ imply
the following result.

\begin{proposition}
If $\mu_1, \widetilde{\mu}_1 \leq n$, then
\begin{equation}\label{duality}
e^*_k(\widetilde{\mu}) = h^*_k(\mu),
\end{equation}
that is, the eigenvalue of $\mathbf{H}_k(n)$ on the Schur module of shape $\mu$ equals
the eigenvalue of $\mathbf{I}_k(n)$ on the Schur module of shape $\widetilde{\mu}$.
\end{proposition}
Notice that the following {\textit{Duality Theorem}}  is an immediate consequence of the preceding
Proposition.

\begin{theorem}\label{finite duality}
Let $\mu$ be such that $\mu_1, \widetilde{\mu}_1 \leq n.$
For every $\boldsymbol{\varrho} \in \boldsymbol{\zeta}(n)$ the eigenvalue of $\boldsymbol{\varrho}$ on the
$gl(n)$-irreducible module $Schur_{\widetilde{\mu}}(n)$ (with highest weight $\mu$) equals eigenvalue of
$\mathcal{W}_n \Big( \boldsymbol{\varrho} \Big)$ on the
$gl(n)$-irreducible module $Schur_{\mu}(n)$ (with highest weight $\widetilde{\mu}$).
\end{theorem}

The preceding result, in combination with the characterization results of 
subsection \ref{Characterization Theorems}, implies

\begin{corollary}\label{Schur duality}

Let $\widetilde{\lambda}_1, \lambda_1  \leq n$. Then

$$
\mathcal{W}_n \Big(  \mathbf{S}_{\lambda}(n)  \Big) = \mathbf{S}_{\widetilde{\lambda}}(n).
$$
\end{corollary}
\begin{proof}
We have:
\begin{align*}
&\textrm{If} \ |\mu| < |\lambda|, \  \textrm{then}  &   \mathbf{S}_{\lambda}(n) 
\cdot v_{\mu} &= 0,
\\
&\textrm{If} \ |\mu| = |\lambda|, \ \textrm{then} &
\mathbf{S}_{\lambda} (n) \cdot v_{\mu} &= \delta_{\lambda, \mu} \c H(\lambda) \ v_{\mu}.
\end{align*}
On the other hand, from Theorem \ref{finite duality}, it follows:
\begin{align*}
&\textrm{If} \ |\widetilde{\mu}|  < |\widetilde{\lambda}|, \  \textrm{then}  &
 \mathcal{W}_n \Big(  \mathbf{S}_{\lambda}(n)  \Big) \cdot v_{\widetilde{\mu}} =  0,
\\
&\textrm{If} \ |\widetilde{\mu}|  = |\widetilde{\lambda}|, \ \textrm{then} &
\mathcal{W}_n \Big(  \mathbf{S}_{\lambda}(n)  \Big) \cdot v_{\widetilde{\mu}} = \delta_{\lambda, \mu} \c H(\lambda) \cdot v_{\widetilde{\mu}}.
\end{align*}
Since $\delta_{\lambda, \mu} \c H(\lambda)  =
\delta_{\widetilde{\lambda}, \widetilde{\mu}} \c H(\widetilde{\lambda})$,
the assertion follows from Propositions \ref{Char Schur shift} and \ref{SahiOkounkv}.
\end{proof}

Since 
$
\mathbf{I}_k(n) = \mathbf{S}_{(k)}(n)
$ and
$ 
\mathbf{H}_k(n)  = \mathbf{S}_{(1^k)}(n), 
$ 
then
$$
\mathcal{W}_n \Big(  \mathbf{I}_k(n)  \Big) = \mathbf{H}_k(n), \quad k = 1, 2, \ldots, n.
$$
by Corollary \ref{Schur duality}.
 
\begin{corollary}
The algebra automorphism $\mathcal{W}_n$ is an involution.
\end{corollary}

\section{The limit $n \rightarrow  \infty$ for $\boldsymbol{\zeta}(n)$: the algebra $\boldsymbol{\zeta}$}\label{zeta limit}

\subsection{The  monomorphisms $\mathbf{i}_{n+1,n}$
and the epimorphisms $\boldsymbol{\pi}_{n,n+1}$}\label{Capelli monomorphisms}

Given $n \in \mathbb{Z}^+$, let  
$$
\mathbf{H}_k(n), \quad k =1, \ldots , n
$$
be the  Capelli free generators  of the center $\boldsymbol{\zeta}(n)$
of the enveloping algebra $\mathbf{U}(gl(n))$, for every $n \in \mathbb{Z}^+$.

For every $n \in \mathbb{Z}^+$, let
$$
\mathbf{i}_{n+1,n} : \boldsymbol{\zeta}(n) \hookrightarrow \zeta(n+1)
$$
be the algebra monomorphism
$$
\mathbf{i}_{n+1,n} : \mathbf{H}_k(n) \rightarrow \mathbf{H}_k(n+1), \quad k = 1,2, \ldots, n.
$$

Given $m \in \mathbb{Z}^+$, let $\boldsymbol{\zeta}(n)^{(m)}$ denote the $m$-th filtration element
of $\boldsymbol{\zeta}(n)$ (with respect to the filtration induced  by the standard filtration of $\mathbf{U}(n)$).
Clearly, the  monomorphisms $\mathbf{i}_{n+1,n}$ are {\it{morphisms in  the category of filtered algebras}}, that is
$$
\mathbf{i}_{n+1,n}\Big[ \boldsymbol{\zeta}(n)^{(m)} \Big] \subseteq \boldsymbol{\zeta}(n+1)^{(m)}
$$

We  consider the {\it{direct limit}} (in the category of filtered algebras):
\begin{equation}\label{direct limit}
\underrightarrow{lim} \ \boldsymbol{\zeta}(n) = \boldsymbol{\zeta}.
\end{equation}

The algebra $\boldsymbol{\zeta}$ inherits a structure of filtered algebra, where
$$
{\boldsymbol{\zeta}}^{(m)} = \underrightarrow{lim} \ \boldsymbol{\zeta}(n)^{(m)}.
$$

On the other hand, given $n \in \mathbb{Z}^+$, we may consider the algebra epimorphism
$$
\boldsymbol{\pi}_{n,n+1} : \boldsymbol{\zeta}(n+1) \twoheadrightarrow \boldsymbol{\zeta}(n),
$$
such that
$$
\boldsymbol{\pi}_{n,n+1}(\mathbf{H}_k(n+1)) = \mathbf{H}_k(n)\, \quad k = 1, 2, \ldots, n,
$$
$$
\boldsymbol{\pi}_{n,n+1}(\mathbf{H}_{n+1}(n+1)) = 0.
$$

The following Propositions are fairly obvious from the definitions.
\begin{proposition} We have

\begin{enumerate}

\item
$
Ker \big( \boldsymbol{\pi}_{n,n+1} \big) = \Big(  \mathbf{H}_{n+1}(n+1) \Big),
$
the bilateral ideal of $\boldsymbol{\zeta}(n+1)$ generated by the element $\mathbf{H}_{n+1}(n+1)$.

\item

 The  epimorphism $\boldsymbol{\pi}_{n,n+1}$ is the (filtered) left inverse  of the  monomorphism 
$\mathbf{i}_{n+1,n}.$
In symbols,
$$
 \boldsymbol{\pi}_{n,n+1} \circ \mathbf{i}_{n+1,n} = Id_{\boldsymbol{\zeta}(n)}.
$$
\end{enumerate}

\end{proposition}

\begin{proposition}\label{inverso filtrato}

If $n \geq m  $ , then the restriction  $\boldsymbol{\pi}^{(m)}_{n,n+1}$ of $\boldsymbol{\pi}_{n,n+1}$
to  $\boldsymbol{\zeta}(n+1)^{(m)}$
and the restriction  $\mathbf{i}^{(m)}_{n+1,n}$ of  $\mathbf{i}_{n+1,n}$ to  $\boldsymbol{\zeta}(n)^{(m)}$
 are  the  inverse  of  each  other.

\end{proposition}

The crucial point is that the projections $\boldsymbol{\pi}_{n,n+1}$ admit an {\it{intrinsic/invariant}} presentation
that is founded upon the {\it{Olshanski decomposition}}.

\subsection{The Olshanski decomposition/projection }\label{The Olshanski decomposition/projection}

We recall a special case of an essential costruction due to Olshanski \cite{Olsh1-BR}, \cite{Olsh3-BR}.
For the sake of simplicity, we follow Molev (\cite{Molev1-BR}, pp. 928 ff.).

Let $\mathbf{U}(gl(n+1))^0$ be the centralizer in $\mathbf{U}(gl(n+1))$ of the element $e_{n+1,n+1}$ of the standard basis
of $gl(n+1)$, regarded as an element of $\mathbf{U}(gl(n+1))$.

Let $\mathcal{I}(n+1)$ be the {\it{left ideal}} of $\mathbf{U}(gl(n+1))$ generated by the elements
$$
e_{i,n+1}, \quad i = 1, 2, \ldots, n+1.
$$

Let $\mathcal{I}(n+1)^0$  be the intersection
\begin{equation}\label{first Olshanski decomposition}
\mathcal{I}(n+1)^0 = \mathcal{I}(n+1) \cap \mathbf{U}(gl(n+1))^0.
\end{equation}

We recall that $\mathcal{I}(n+1)^0$ is a {\it{bilateral ideal}} of $\mathbf{U}(gl(n+1)^0$, and  the following {\it{direct sum decomposition}} hold

\begin{equation}\label{direct sum decomposition}
\mathbf{U}((gl(n+1))^0 = \mathbf{U}(gl(n)) \oplus 
\mathcal{I}(n+1)^0.
\end{equation}

Therefore, the {\it{Olshanski map}}
$$
\mathcal{M}_{n+1} : \mathbf{U}((gl(n+1))^0 \twoheadrightarrow \mathbf{U}(gl(n))
$$
that maps any element in the direct summand  $\mathbf{U}(gl(n))$ to itself and any element
in the direct summand  $\mathcal{I}(n+1)^0$ to zero is a well-defined algebra epimorphism.

Since $\boldsymbol{\zeta}(n+1)$ is a subalgebra of $\mathbf{U}((gl(n+1))^0$, the direct sum decomposition (\ref{direct sum decomposition})
induces a  direct sum decomposition
of any element in $\boldsymbol{\zeta}(n+1)$ and the $\mathcal{M}_{n+1}$ map defines, by restriction, an algebra epimorphism
$$
\boldsymbol{\mu}_{n,n+1} : \boldsymbol{\zeta}(n+1) \twoheadrightarrow \boldsymbol{\zeta}(n).
$$

 In plain words, any element $\boldsymbol{\varrho} \in \boldsymbol{\zeta}(n+1)$ admits a {\it{unique}} decomposition

 \begin{equation}\label{Olshanski decomposition}
 \boldsymbol{\varrho} = \boldsymbol{\varrho}' \dotplus \boldsymbol{\varrho}^0, \quad \boldsymbol{\varrho}' \in \boldsymbol{\zeta}(n), \
 \boldsymbol{\varrho}^0 \in \mathcal{I}(n+1)^0.
 \end{equation}

We call the decomposition (\ref{Olshanski decomposition}) the {\it{Olshanski decomposition}} of the
element $\boldsymbol{\varrho} \in \boldsymbol{\zeta}(n+1)$.

In this notation, the projection
$$
\boldsymbol{\mu}_{n,n+1} : \boldsymbol{\zeta}(n+1) \twoheadrightarrow \boldsymbol{\zeta}(n),
$$
$$
\boldsymbol{\mu}_{n+1,n}(\boldsymbol{\varrho}) = \boldsymbol{\varrho}', \quad \boldsymbol{\varrho} \in \boldsymbol{\zeta}(n+1)
$$
is defined.

\begin{proposition}
We have

\begin{enumerate}

\item if $k \leq n$, then
$$
\mathbf{H}_k(n+1) = \mathbf{H}_k(n) \dotplus \mathbf{H}_k(n+1)^0,
$$
where
$$
\mathbf{H}_k(n+1)^0 = \mathbf{H}_k(n+1) - \mathbf{H}_k(n) \in \mathcal{I}(n+1)^0,
$$
and
$$
\mathbf{H}_k(n) \in \boldsymbol{\zeta}(n);
$$
\item 
$
\mathbf{H}_{n+1}(n+1) = \mathbf{H}_{n+1}(n+1)^0.
$
\end{enumerate}
\end{proposition}

\begin{example}We have:
\begin{align*}
\mathbf{H}_2(4) &= [21|12]+[31|13]+[41|14]+[32|23]+[42|24]+[43|34]
\\
&= \mathbf{H}_2(3) \dotplus \mathbf{H}_2(4)^0,
\end{align*}
where
$$
\mathbf{H}_2(3) = [21|12]+[31|13]+[32|23] \in \boldsymbol{\zeta}(3),
$$
and
$$
\mathbf{H}_2(4)^0 = [41|14]+[42|24]+[43|34] \in \mathcal{I}(4)^0.
$$
\end{example}\qed

\begin{corollary}\label{Olshanski Capelli} We have

\begin{enumerate}

\item if $k \leq n$, then
$
\boldsymbol{\mu}_{n,n+1} \big(\mathbf{H}_k(n+1)\big) = \mathbf{H}_k(n),
$

\item 
$
\boldsymbol{\mu}_{n,n+1} \big(\mathbf{H}_{n+1}(n+1) \big) = 0.
$
\end{enumerate}

\end{corollary}

\begin{proposition}\label{pi=mu}

The map $\boldsymbol{\mu}_{n,n+1}$ is the {\it{same}} as the map $\boldsymbol{\pi}_{n,n+1}$.

\end{proposition}
\begin{proof} The family 
$$
\big\{ \mathbf{H}_1(n+1), \mathbf{H}_2(n+1), \ldots, \mathbf{H}_n(n+1), \mathbf{H}_{n+1}(n+1) \big\}        
$$ 
is a system of algebraically independent generators of the algebra $\boldsymbol{\zeta}(n+1)$.
From Proposition \ref{Olshanski Capelli}, we obtain:
\begin{itemize}

\item [--] if $k \leq n$, then
$$
\boldsymbol{\mu}_{n,n+1}(\mathbf{H}_k(n+1)) = \mathbf{H}_k(n) = \boldsymbol{\pi}_{n,n+1}(\mathbf{H}_k(n+1));
$$
\item [--] 
$
\boldsymbol{\mu}_{n,n+1}(\mathbf{H}_{n+1}(n+1)) = 0 = \boldsymbol{\pi}_{n,n+1}(\mathbf{H}_{n+1}(n+1)).
$
\end{itemize}
\end{proof}

In the following, we refer to the projections
$$\boldsymbol{\mu}_{n,n+1} = \boldsymbol{\pi}_{n,n+1}$$
as the {\it{Capelli-Olshanski projections}}.

From Proposition \ref{inverso filtrato}, the algebra $\boldsymbol{\zeta}$ (direct limit) is the same as the
{\it {inverse limit in the category of filtered algebras}}
\begin{proposition}\label{projective limit} We have
$$
\boldsymbol{\zeta} = \underleftarrow{lim} \ \boldsymbol{\zeta}(n)
$$
with respect to the system of Capelli-Olshanski projections.

\end{proposition}

\subsection{Main results}\label{main}

From Theorem \ref{Capelli generators} and Proposition \ref{Olshanski Capelli},
we infer

\begin{proposition}\label{Olshanski altri}

\begin{enumerate} We have

\item $\mathbf{I}_k(n+1) = \mathbf{I}_k(n) \dotplus \mathbf{I}_k(n+1)^0,$
where
$$
\mathbf{I}_k(n+1)^0 = \mathbf{I}_k(n+1) - \mathbf{I}_k(n) \in \mathcal{I}_k(n+1)^0,
$$
and
$$
\mathbf{I}_k(n) \in \boldsymbol{\zeta}(n).
$$
Then
\begin{equation}
\boldsymbol{\pi}_{n,n+1}(\mathbf{I}_k(n+1)) = \mathbf{I}_k(n).
\end{equation}

\item $\mathbf{S}_{\lambda}(n+1) = \mathbf{S}_{\lambda}(n) \dotplus \mathbf{S}_{\lambda}(n+1)^0,$
where
$$
\mathbf{S}_{\lambda}(n+1)^0 = \mathbf{S}_{\lambda}(n+1) - \mathbf{S}_{\lambda}(n) \in \mathcal{I}_k(n+1)^0,
$$
and
$$
\mathbf{S}_{\lambda}(n) \in \boldsymbol{\zeta}(n).
$$
Then
\begin{equation}
\boldsymbol{\pi}_{n,n+1}(\mathbf{S}_{\lambda}(n+1)) = \mathbf{S}_{\lambda}(n).
\end{equation}

\end{enumerate}

\end{proposition}

By combining the preceding Proposition with Proposition \ref{inverso filtrato}, we get

\begin{theorem}\label{germs}We have:
\begin{enumerate}

\item Given a positive integer $k$, if $n \geq k$
then
$$
\mathbf{i}_{n+1, n}(\mathbf{I}_k(n)) = \mathbf{I}_k(n+1);
$$

\item Given a partition $\lambda$, if $n \geq  |\lambda|$
then
$$
\mathbf{i}_{n+1, n}(\mathbf{S}_{\lambda}(n)) = \mathbf{S}_{\lambda}(n+1).
$$
\end{enumerate}
\end{theorem}

Passing to the direct limit $\underrightarrow{lim} \ \boldsymbol{\zeta}(n) = \boldsymbol{\zeta},$ we  set:

\begin{enumerate}\label{series} 
\item $\mathbf{H}_k  \stackrel{def}{=} \underrightarrow{lim}  
\ \mathbf{H}_k(n) \in \boldsymbol{\zeta}$.

\item $\mathbf{I}_k \stackrel{def}{=} \underrightarrow{lim} \  \mathbf{I}_k(n)  \in \boldsymbol{\zeta}$.

\item $\mathbf{S}_{\lambda} \stackrel{def}{=} \underrightarrow{lim} \ 
\mathbf{S}_{\lambda}(n)  \in \boldsymbol{\zeta}$.

\end{enumerate}

From the definition of the monomorphisms $\mathbf{i}_{n+1, n}$ and Theorem \ref{germs}, 
the elements
$\mathbf{H}_{k}, \mathbf{I}_{k},  \mathbf{S}_{\tilde{\lambda}} \in \boldsymbol{\zeta}$
can be consistently written as {\emph{formal infinite sums}}.

\begin{proposition}\label{formal series} We have

\begin{enumerate} 
\item  
$$
\mathbf{H}_k =
 \sum_{i_1 < \cdots < i_k } \ [ i_k \cdots i_2 i_1 | i_1 i_2 \cdots i_k ],
$$
where the sum is extended to all increasing $k$-tuples $i_1 < i_2 < \cdots < i_k$ in $ \mathbb{Z}^+$.

\item 
$$
\mathbf{I}_k  =
 \sum_{j_1 < j_2 < \cdots < j_p } \ (i_{j_1}! \ i_{j_2}! \cdots i_{j_p}!)^{-1}
  \ [j_p^{i_{j_p}} \cdots j_2^{i_{j_2}} j_1^{i_{j_1}} | j_1^{i_{j_1}} j_2^{i_{j_2}}  \cdots j_p^{i_{j_p}} ]^{*},
$$
 where 
the sum is extended to all $p$-tuples $j_1 < j_2 < \cdots < j_p$ in $ \mathbb{Z}^+$,
and to all the $p$-tuples of exponents $(i_{j_1}, i_{j_2}, \cdots, i_{j_p})$ such that
$$i_{j_1}+i_{j_2}+ \cdots  + i_{j_p} = k.$$

\item 
$$
\mathbf{S}_{\lambda}(n) =
 \frac {1} {H(\tilde{\lambda})} \ \  \sum_S \ [\ \fbox{$S \ | \ S$}\ ], \label{Schur basis2}
$$
where the sum is extended to all  row-increasing tableaux $S$ of shape $\tilde{\lambda}$ on the 
alphabet $\mathbb{Z}^+$.

\end{enumerate}

\end{proposition}

From Proposition \ref{projective limit}, it follows

\begin{corollary} We have:

\begin{enumerate}
\item
$
\underleftarrow{lim}  \ \mathbf{H}_k(n) = \mathbf{H}_k \in \boldsymbol{\zeta},
$

\item
$
\underleftarrow{lim} \  \mathbf{I}_k(n) = \mathbf{I}_k \in \boldsymbol{\zeta},
$

\item
$
 \underleftarrow{lim} \ \mathbf{S}_{\lambda}(n) = \mathbf{S}_{\lambda} \in \boldsymbol{\zeta}.
$

\end{enumerate}

\end{corollary}

Due the fact that the algebra $\boldsymbol{\zeta}$ is defined as a direct limit, we infer:

\begin{theorem} $ $

\begin{enumerate}

\item
The set
$$
\Big\{  \mathbf{H}_k; \  k \in \mathbb{Z}^+ \Big\}
$$
is a system of free algebraic generators of $\boldsymbol{\zeta}$.

\item
The set
$$
\Big\{  \mathbf{I}_k; \  k \in \mathbb{Z}^+ \Big\}
$$
is a system of free algebraic generators of $\boldsymbol{\zeta}$.

\item
The set
$$
\Big\{ \mathbf{S}_{\lambda}; \  \lambda \ any \ partition \ \Big\}
$$
is a linear basis  of $\boldsymbol{\zeta}$.

\end{enumerate}

\end{theorem}

\subsection{Duality in $\boldsymbol{\zeta}$}

Let
$$
\mathcal{W} : \boldsymbol{\zeta} \rightarrow \boldsymbol{\zeta}
$$
denote the automorphism such that
$$
\mathcal{W} \Big( \mathbf{H}_k \Big) = \mathbf{I}_k, \quad   for \ every  \ k \in \mathbb{Z}^+.
$$

Since $\underrightarrow{lim} \ \boldsymbol{\zeta}(n) = \boldsymbol{\zeta}$, 
Corollary \ref{Schur duality} implies

\begin{theorem}

\

\begin{enumerate}

\item
For every partition $\lambda$,
 $$
  \mathcal{W} \Big( \mathbf{S}_{\lambda} \Big) = \mathbf{S}_{\tilde{\lambda}}.
 $$

\item
In particular,
$$
  \mathcal{W} \Big( \mathbf{I}_k \Big) = \mathbf{H}_k, \quad for \ every \ k \in\mathbb{Z}^+;
$$
then, the automorphisms $\mathcal{W}$ is an involution.

\end{enumerate}
\end{theorem}

\section{The algebra $\Lambda^*(n)$ of shifted symmetric polynomials and the 
Harish-Chandra Isomorphism}\label{Lambda(n)}

In this subsection we follow  Okounkov and Olshanski \cite{OkOlsh-BR}.

The algebra
$\Lambda^*(n)$ of {\textit{shifted symmetric polynomials}} is the algebra
of polynomials $p(x_1, x_2, \ldots, x_n)$  that satisfy
 the {\textit{shifted symmetry}} condition:
$$
 p(x_1, \ldots , x_i, x_{i+1}, \ldots, x_n) = p(x_1, \ldots , x_{i+1} - 1, x_i + 1,
 \ldots, x_n),
$$
for $i = 1, 2, \ldots, n - 1.$

The {\textit{Harish-Chandra isomorphism}} $\chi_n$ is the algebra isomorphism
$$
\chi_n : \boldsymbol{\zeta}(n) \longrightarrow \Lambda^*(n), \qquad  \ A \mapsto \chi_n(A),
$$
where $\chi_n(A)$ is the shifted symmetric polynomial such that, for every highest weight module $V_{\mu}$,
the evaluation $\chi_n(A)(\mu_1, \mu_2, \ldots , \mu_n)$ equals the eigenvalue of
$A \in \boldsymbol{\zeta}(n)$ in  $V_{\mu}$ (see, e.g. \cite{OkOlsh-BR}).

From  Corollary \ref{Capelli eigenvalues}.$1$, it follows

\begin{proposition}\label{elementary}
\begin{align*}
\chi_n(\mathbf{H}_k(n)) =& \ \mathbf{e}_k^{*}(x_1, x_2, \ldots, x_n) 
\\
= & \sum_{1 \leq i_1 < i_2 < \cdots < i_r \leq n} \ (x_{i_1}  + k  - 1)
(x_{i_2}  + k - 2) \cdots (x_{i_k})
\end{align*}
for every $k =  1, 2, \ldots, n.$
\end{proposition}
Clearly, $\chi_n(\mathbf{H}_0(n)) = \mathbf{e}_0^{*}(x_1, x_2, \ldots, x_n) = \mathbf{1}$.

The polynomials $\mathbf{e}_k^{*}(x_1, x_2, \ldots, x_n) \in \Lambda^*(n)$ are the
\emph{elementary} shifted symmetric polynomials.

From Theorem \ref{vertical strip}.$2$, it follows

\begin{proposition}\label{shifted complete-BR}
\begin{align*}
\chi_n(\mathbf{I}_k(n)) =& \ \mathbf{h}_k^{*}(x_1, x_2, \ldots, x_n) 
\\
= & \sum_{1 \leq i_1 \leq i_2 < \cdots \leq i_k \leq n} \ (x_{i_1}  - k  + 1)
(x_{i_2}  - k + 2) \cdots (x_{i_k}),
\end{align*}
for every $k =  1, 2, \ldots, n.$
\end{proposition}
Clearly, $\chi_n(\mathbf{I}_0(n)) = \mathbf{h}_0^{*}(x_1, x_2, \ldots, x_n) = \mathbf{1}$.

The polynomials $\mathbf{h}_k^{*}(x_1, x_2, \ldots, x_n) \in \Lambda^*(n)$ are the
\emph{complete} shifted symmetric polynomials.

Recall that, given a variable $z$ and a natural integer $p$, the symbol $(z)_p$ denotes
 the {\it{falling factorial polynomial}}:
$$
(z)_p = z(z - 1) \cdots (z - p + 1), \quad p \geq 1,
\quad \quad
(z)_0 = 1.
$$

Let $\mu$ be a partition, $\widetilde{\mu}_1 \leq n$.

Following \cite{OkOlsh-BR}, consider the polynomial
\begin{align}
\mathbf{s}^*_\lambda(x_1, \ldots, x_n) =&
\frac { det \Big[ (x_i + n -i)_{\lambda_i + n - j } \Big]} { det \Big[ (x_i + n -i)_{n - j} \Big]}
\\
=& \sum_ {T \in RSSYT(\lambda)} \ (x_{T(s)} - c(s)),
\end{align}
where $RSSYT(\mu)$ denotes the set of all \emph{reverse semistandard}  
\footnote{A  Young tableau whose entries belong to $\{1, . . ., n \}$ 
and weakly decrease from left to right along
each row and strictly decrease down each column.} Young tableaux $T$
of shape $\lambda$ over the set $\{1, 2, \ldots,n \}$, $T(s)$ denotes the symbol of 
in the cell $s$ of the Ferrers diagram of $\mu$ and $c(s) = j - i$ is the content
of the cell $s$ in position $(i, j)$.

The polynomials $\mathbf{s}^*_\mu(x_1, \ldots, x_n) \in \Lambda^*(n)$ are the
{\it{shifted  Schur}} polynomials.

From the Characterization Theorem for the Schur elements $\mathbf{S}_\lambda(n) \in \boldsymbol{\zeta}(n)$
(see subsection \ref{Characterization  Theorems})
and the  Characterization Theorem for the shifted Schur polynomials \cite{OkOlsh-BR}, we have:

\begin{theorem}
For every  $\lambda$, 
$\widetilde{\lambda}_1 \leq n$,
$$  \chi_n( \mathbf{S}_\lambda(n) ) = \mathbf{s}^*_\lambda(x_1, \ldots, x_n).$$
\end{theorem}

 From Theorem \ref{Capelli generators} and Proposition \ref{elementary}, it follows

\begin{proposition}
\begin{enumerate}
\
\item The set
$$
\Big\{  \mathbf{e}_k^{*}(x_1, x_2, \ldots, x_n); \ k = 1, 2, \ldots, n   \Big\}
$$
is a set of free algebra generators of the polynomial algebra $\Lambda^*(n)$.

\item The set
$$
\Big\{  \mathbf{h}_k^{*}(x_1, x_2, \ldots, x_n); \ k = 1, 2, \ldots, n   \Big\}
$$
is a set of free algebra generators of the polynomial algebra $\Lambda^*(n)$.

\item The set
$$
\Big\{  \mathbf{s}^*_\lambda(x_1, \ldots, x_n); \ \widetilde{\lambda}_1 \leq n   \Big\}
$$
is a linear basis of the polynomial algebra $\Lambda^*(n)$.

\end{enumerate}
\end{proposition}

\section{The algebra $\Lambda^*$  of shifted symmetric functions}\label{Lambda}

Let

$$
\mathbf{i}^*_{n+1,n} : \Lambda^*(n) \hookrightarrow \Lambda^*(n+1)
$$
be the algebra monomorphism such that
$$
\mathbf{i}^*_{n+1,n}\big( \mathbf{e}_k^{*}(x_1, x_2, \ldots, x_n) \big) = 
\mathbf{e}_k^{*}(x_1, x_2, \ldots, x_n, x_{n+1}),
$$
for $k = 1, 2, \ldots, n.$

Given $m \in \mathbb{Z}^+$, let ${\Lambda^*(n)}^{(m)}$ denote the $m$-th filtration element
of $\Lambda^*(n)$ (with respect to the filtration induced  by the standard filtration of the algebra of polynomials
in the variables $x_1, x_2, \ldots, x_n$).

Clearly, the  monomorphisms $\mathbf{i}^*_{n+1,n}$ are {\it{morphisms in  the category of filtered algebras}}, 
that is
$$
\mathbf{i}^*_{n+1,n}\Big[ \Lambda^*(n)^{(m)} \Big] \subseteq \Lambda^*(n+1)^{(m)}.
$$

The algebra of \emph{shifted symmetric functions} $\Lambda^*$ 
is the {\it{direct limit}} (in the category of filtered algebras):
\begin{equation}\label{direct limit}
\Lambda^* = \underrightarrow{lim} \ \Lambda^*(n).
\end{equation}

The algebra $\Lambda^*$ inherits a structure of filtered algebra, where
$$
{\Lambda^*}^{(m)} = \underrightarrow{lim} \ \Lambda^*(n)^{(m)}.
$$

Let
$$
\boldsymbol{\pi}^*_{n,n+1} : \Lambda^*(n+1) \twoheadrightarrow \Lambda^*(n)
$$
be the algebra epimorphism such that
$$
\boldsymbol{\pi}^*_{n,n+1}\big( \mathbf{e}_k^{*}(x_1, x_2, \ldots, x_n, x_{n+1}) \big) = \mathbf{e}_k^{*}(x_1, x_2, \ldots, x_n),
$$
for $k = 1, 2, \ldots, n,$ and
$$
\boldsymbol{\pi}^*_{n,n+1}\big( \mathbf{e}_{n+1}^{*}(x_1, x_2, \ldots, x_n, x_{n+1}) \big) = 0.
$$
Clearly
$$
\boldsymbol{\pi}^*_{n,n+1}\big( \mathbf{f}^{*}(x_1, x_2, \ldots, x_n, x_{n+1}) \big) = \mathbf{f}^{*}(x_1, x_2, \ldots, x_n, 0),
$$
for every $\mathbf{f}^{*}(x_1, x_2, \ldots, x_n, x_{n+1}) \in \Lambda^*(n+1).$

As for the centers $\boldsymbol{\zeta}(n+1)$ and $\boldsymbol{\zeta}(n)$, 
the following Remarks and Proposition on $\Lambda^*(n+1)$ and
$\Lambda^*(n)$ are obvious from the definitions.

\begin{proposition}\label{filtrato shift}
We have:
\begin{enumerate}

\item
$Ker \big( \boldsymbol{\pi}^*_{n,n+1} \big)$
is the bilateral ideal
$$
\Big(  \mathbf{e}_{n+1}^{*}(x_1, x_2, \ldots, x_n, x_{n+1}) \Big)
$$
of $\Lambda^*(n+1)$
generated by the element $\mathbf{e}_{n+1}^{*}(x_1, x_2, \ldots, x_n, x_{n+1})$.

\item
The  projection $\boldsymbol{\pi}^*_{n,n+1}$ is the left inverse  of the  monomorphism $\mathbf{i}^*_{n+1,n}.$
In symbols,
$$
 \boldsymbol{\pi}^*_{n,n+1} \circ \mathbf{i}^*_{n+1,n} = Id_{\Lambda^*(n)}.
$$

\item
If $m \leq n  $, then the restriction  ${\boldsymbol{\pi}^*_{n,n+1}}^{(m)}$ of $\boldsymbol{\pi}^*_{n,n+1}$
to  ${\Lambda^*(n+1)}^{(m)}$
and the restriction  $\mathbf{i}^*_{n+1,n}$ of  $\mathbf{i}^*_{n+1,n}$ to  ${\Lambda^*(n)}^{(m)}$
 are  the  inverse  of  each  other.

\end{enumerate}

\end{proposition}

From Proposition \ref{filtrato shift}, the algebra $\Lambda^*$ (direct limit) is the same as the
inverse limit in the category of filtered algebras
$$
\Lambda^* = \underleftarrow{lim} \ \Lambda^*(n)
$$
with respect to the system of the projections $\boldsymbol{\pi}^*_{n,n+1}$, and therefore, the algebra $\Lambda^*$
is the algebra of shifted symmetric functions of \cite{OkOlsh-BR}.

Consider the  commutative diagram:

\begin{equation}\label{diagramma commutativo}
\begin{tikzpicture}[description/.style={fill=white,inner sep=2pt}]
\fontsize{15.0}{20.0}
\matrix (m) [matrix of math nodes, row sep=5em,
column sep=9em, text height=3.0ex, text depth=1.0ex]
{
\textbf{$\boldsymbol{\zeta}^{(m)}(n)$}
&
\textbf{$\boldsymbol{\zeta}^{(m)}(n+1)$}
\\
\textbf{$\Lambda ^{*(m)}(n)$}
&
\textbf{$\Lambda ^{*(m)}(n+1)$}
\\
};

\fontsize{12}{14}

\draw [line width=0.02cm,<<-] (-2.6,-1.7) -- (1.7,-1.7);
\node  [above] at (0,-1.7) {${\boldsymbol{\pi}^*_{n,n+1}}$};
\draw [line width=0.02cm,right hook->] (-2.6,-2.0) -- (1.7,-2.0);
\node  [below] at (0,-2.0) {${\boldsymbol{i}^*_{n+1,n}}$};

\draw [line width=0.02cm, <<-] (-2.6,1.9) -- (1.7,1.9);
\node  [above] at (0,1.9) {${\boldsymbol{\pi}_{n,n+1}}$};
\draw [line width=0.02cm,right hook->] (-2.6,1.6) -- (1.7,1.6);
\node  [below] at (0,1.6) {${\boldsymbol{i}_{n+1,n}}$};

\draw [line width=0.02cm, ->] (-3.7,1.4) -- (-3.7,-1.4);
\node   at (-3.3,0.0) {{$\chi_n$}};
\draw [line width=0.02cm, ->] (3.2,1.4) -- (3.2,-1.4);
\node   at (3.8,0.0) {{$\chi_{n+1}$}};

\end{tikzpicture}
\end{equation}

\begin{theorem}\label{isomorfismo HC infinito}
If $m \leq n  $, the pairs of horizontal arrows in the commutative diagram (\ref{diagramma commutativo}) denote
mutually inverse \emph{isomorphisms}.
\end{theorem}

Passing to the direct limit, we get the isomorphism of filtered algebras:
$$
    \chi : \boldsymbol{\zeta} \rightarrow  \Lambda^*.
$$

Given $\boldsymbol{\varrho} = \underrightarrow{lim} \ \boldsymbol{\varrho}(n) \in \boldsymbol{\zeta}^{*(m)}$
and every partion $\mu$,
if
$$
n \geq  max\{ m,  \ l(\widetilde{\mu}) = \mu_1 \},
$$
then
$$
\chi_n(\boldsymbol{\varrho}(n))(\widetilde{\mu}) =  \chi_{n+1} \big( {\mathbf{i}^*_{n+1,n}}^{(m)} (\boldsymbol{\varrho}(n)) \big)  (\widetilde{\mu}) = \chi_{n+1}(\boldsymbol{\varrho}(n+1))(\widetilde{\mu}).
$$

Therefore, the sequence
$$
\Big( \ \chi_{n} \big( (\boldsymbol{\varrho}(n)) \big)  (\widetilde{\mu}) \ \Big)_{n \in \mathbb{N}^+}
$$
is \emph{definitively constant} and
the eigenvalue
\begin{equation}\label{universal HarishChandra}
\chi(\boldsymbol{\varrho})(\widetilde{\mu}) = \chi_n(\boldsymbol{\varrho}(n))(\widetilde{\mu}), 
\end{equation}
$n$ sufficiently  large, is well-defined.

\begin{corollary}
\begin{enumerate}
\
\item
For every $k \in \mathbb{Z}^+$,
$$
\chi \big( \mathbf{H}_k \big) = \mathbf{e}^*_k \in \Lambda^*,
$$
where
$$
\mathbf{e}^*_k = \sum_{ i_1 < i_2 < \cdots < i_k } \ (x_{i_1}  + k  - 1)
(x_{i_2}  + k - 2) \cdots (x_{i_k}), \quad i_s \in \mathbb{Z}^+,
$$
$\mathbf{e}^*_k$ the $k$-th elementary shifted symmetric function;

\item
For every $k \in \mathbb{Z}^+$,
$$
\chi \big( \mathbf{I}_k \big) = \mathbf{h}^*_k \in \Lambda^*,
$$
where
$$
\mathbf{h}_k^{*}
= \sum_{i_1 \leq i_2 < \cdots \leq i_k } \ (x_{i_1}  - k  + 1)
(x_{i_2}  - k + 2) \cdots (x_{i_k}), \quad i_s \in \mathbb{Z}^+,
$$
$\mathbf{h}^*_k$ the $k$-th complete shifted symmetric function.

\end{enumerate}

\end{corollary}

Since $\chi_n \big(\mathbf{S}_{\lambda}(n)\big) = s^*_{\lambda}(n)$, from Proposition 
\ref{Olshanski altri}, item $2)$ and Theorem \ref{germs}, item  $2)$ we have

\begin{corollary}
\begin{enumerate}

\

\item (stability property  \cite{OkOlsh-BR})
$\boldsymbol{\pi}^*_{n,n+1}(s^*_{\lambda}(n+1)) = s^*_{\lambda}(n) , \quad n \in \mathbb{Z}^+.$

\item If $n \geq |\lambda|$, then
$$\boldsymbol{i}^*_{n+1,n}(s^*_{\lambda}(n)) = s^*_{\lambda}(n+1).$$
\end{enumerate}
\end{corollary}

The shifted symmetric Schur \emph{function} $s^*_{\lambda}$ is the (direct/inverse)  
limit
$$
s^*_{\lambda} = \underrightarrow{lim} (s^*_{\lambda}(n)) = \underleftarrow{lim} (s^*_{\lambda}(n)).
$$

Then
\begin{corollary} For every $\lambda$, we have
\begin{enumerate}

\item
$$
\chi \big(\mathbf{S}_{\lambda}\big) =  s^*_{\lambda}.
$$

\item
$$
\mathbf{s}^*_\lambda =
\sum_ {T \in RSSYT(\lambda)} \ (x_{T(s)} - c(s)),
$$
where $RSSYT(\lambda)$ is the set of all reverse semistandard Young tableaux $T$
of shape $\lambda$ over the set $\mathbb{Z}^+$.
\end{enumerate}
\end{corollary}

Let
$$
\mathcal{W} : \boldsymbol{\zeta} \rightarrow \boldsymbol{\zeta}
$$
denote the automorphism such that
$$
\mathcal{W} \Big( \mathbf{H}_k \Big) = \mathbf{I}_k, \quad   for \ every  \ k \in \mathbb{Z}^+,
$$

and let

$$
w : \Lambda^* \rightarrow \Lambda^*
$$
denote the automorphism such that
$$
w \Big( \mathbf{e}^*_k \Big) = \mathbf{h}^*_k, \quad   for \ every  \ k \in \mathbb{Z}^+.
$$

Clearly, 
$$
 \chi \circ \mathcal{W} = w \circ \chi.
$$

\begin{corollary}

\

\begin{enumerate}

\item
For every partition $\lambda$,
 $$
  w \Big( s^*_{\lambda}\Big) = \mathbf{s^*}_{\widetilde{\lambda}}.
 $$

\item
In particular,
$$
  w \Big( \mathbf{h}^*_k \Big) = \mathbf{e}^*_k, \quad for \ every \ k \in\mathbb{Z}^+;
$$
then, the automorphism $w$ is an \emph{involution}.

\end{enumerate}
\end{corollary}

\end{document}